\documentclass{amsart}
\usepackage{amssymb, amsmath, amsthm, graphics, comment, xspace, enumerate}
\usepackage{color}
\newtheorem{thm}{Theorem}[section]
\newtheorem{cor}[thm]{Corollary}

\newtheorem{prop}[thm]{Proposition}
\theoremstyle{definition}
\newtheorem{defn}[thm]{Definition}
\newtheorem{example}[thm]{Example}
\theoremstyle{remark}
\newtheorem{rem}[thm]{Remark}
\numberwithin{equation}{section}
\begin{document}
\title[Multi-dimensional $c$-almost periodic type functions and applications]{Multi-dimensional $c$-almost periodic type functions and applications}

\author{M. Kosti\' c}
\address{Faculty of Technical Sciences,
University of Novi Sad,
Trg D. Obradovi\' ca 6, 21125 Novi Sad, Serbia}
\email{marco.s@verat.net}

{\renewcommand{\thefootnote}{} \footnote{
\\ \text{  }  \ \    {\it Key words and phrases.} Bohr $({\mathcal B},c)$-almost periodic type functions, $({\mathcal B},c)$-uniformly recurrent type functions, 
abstract Volterra integro-differential equations.
\\  \text{  }  
Marko Kosti\' c is partially supported by grant 451-03-68/2020/14/200156 of Ministry
of Science and Technological Development, Republic of Serbia.
}}

\begin{abstract}
In this paper, we analyze multi-dimensional Bohr $({\mathcal B},c)$-almost periodic type functions.
The main structural characterizations for the introduced classes of Bohr $({\mathcal B},c)$-almost periodic type functions are established. 
Several applications of our abstract theoretical results to
the abstract Volterra integro-differential equations in Banach spaces are provided, as well.
\end{abstract}
\maketitle

\section{Introduction and preliminaries}

The notion of almost periodicity was introduced by the Danish mathematician H. Bohr \cite{h.bor} around 1924-1926 and later generalized by many others (for more details about the subject, we refer the reader to the research monographs \cite{besik}, \cite{diagana}-\cite{gaston}, \cite{nova-mono}-\cite{18}, \cite{pankov} and \cite{30}). Let $I$ be either ${\mathbb R}$ or $[0,\infty),$ and let $f : I \rightarrow X$ be a given continuous function, where $X$ is a complex Banach space equipped with the norm $\| \cdot \|$. Given $\varepsilon>0,$ we call $\tau>0$ a $\varepsilon$-period for $f(\cdot)$ if and only if\index{$\varepsilon$-period}
$
\| f(t+\tau)-f(t) \| \leq \varepsilon,$ $ t\in I.
$
The set consisting of all $\varepsilon$-periods for $f(\cdot)$ is denoted by $\vartheta(f,\varepsilon).$ The function $f(\cdot)$ is said to be almost periodic if and only if for each $\varepsilon>0$ the set $\vartheta(f,\varepsilon)$ is relatively dense in $[0,\infty),$ i.e., 
there exists $l>0$ such that any subinterval of $[0,\infty)$ of length $l$ intersects $\vartheta(f,\varepsilon)$. 

As emphasized in \cite{marko-manuel-ap}, the theory of almost periodic functions of several real variables has not attracted so much attention of the authors by now. In support of our investigation of the multi-dimensional $c$-almost periodicity, we would like to present the following illustrative examples (the notion and notation will be explained in the next section):

\begin{example}\label{nijedanije} (cf. also \cite{marko-manuel-ap})
Suppose that a closed linear operator $A$ generates a strongly continuous semigroup $(T(t))_{t\geq 0}$ on a Banach space $X$ consisting of certain complex-valued functions defined on ${\mathbb R}^{n}.$ Under some assumptions, 
the function
\begin{align*}
u(t,x)=\bigl(T(t)u_{0}\bigr)(x)+\int^{t}_{0}[T(t-s)f(s)](x)\, ds,\quad t\geq 0,\ x\in {\mathbb R}^{n}
\end{align*}
is a unique classical solution of the abstract Cauchy problem
\begin{align*}
u_{t}(t,x)=Au(t,x)+F(t,x),\ t\geq 0,\ x\in {\mathbb R}^{n}; \ u(0,x)=u_{0}(x),
\end{align*} 
where $F(t,x):=[f(t)](x),$ $t\geq 0,$ $x\in {\mathbb R}^{n}.$ For a large class of strongly continuous semigroups (for example, this holds for the Gaussian semigroup on ${\mathbb R}^{n};$ see e.g., \cite[Example 3.7.6]{a43}), there exists a kernel $(t,y)\mapsto E(t,y),$ $t> 0,$ $y\in {\mathbb R}^{n}$ which is integrable on any set $[0,T]\times {\mathbb R}^{n}$ ($T>0$) and satisfies 
$$
[T(t)f(s)](x)=\int_{{\mathbb R}^{n}}F(s,x-y)E(t,y)\, dy,\quad t>0,\ s\geq 0,\ x\in {\mathbb R}^{n}.
$$
Fix a positive real number
$t_{0}>0.$ As in the case that $c=1,$
the $c$-almost periodic behaviour of function $x\mapsto u_{t_{0}}(x)\equiv \int^{t_{0}}_{0}[T(t_{0}-s)f(s)](x)\, ds,$ $x\in {\mathbb R}^{n}$ strongly depends on the $c$-almost periodic behaviour of function $F(t,x)$ in the space variable $x.$ Suppose, for example, that the function $F(t,x)$
is Bohr $c$-almost periodic with respect to the variable $x\in {\mathbb R}^{n},$ uniformly in the variable $t$ on compact subsets of $[0,\infty)$. Then the function 
$u_{t_{0}}(\cdot)$ is also Bohr $c$-almost periodic, which follows from the estimate
\begin{align*}
\Bigl| & u_{t_{0}}(x+\tau)- cu_{t_{0}}(x) \Bigr|\leq  \int^{t_{0}}_{0} \int_{{\mathbb R}^{n}}
 | F(s,x+\tau-y)-cF(s,x-y)| \cdot \bigl|E\bigl(t_{0},y\bigr)\bigr|\, dy\, ds
\\& \leq \epsilon \int^{t_{0}}_{0} \int_{{\mathbb R}^{n}}\bigl|E\bigl(t_{0},y\bigr)\bigr|\, dy\, ds
\end{align*}
and corresponding definitions.
\end{example}

\begin{example}\label{nijedanije121}
In this example, we will observe an interesting feature of the famous d'Alambert formula, which has been used by S. Zaidman \cite[Example 5]{30} in a slightly different context (for almost periodic functions of one real variable).
Let $a>0;$ then it is well known that the regular solution of the wave equation $u_{tt}=a^{2}u_{xx}$ in domain $\{(x,t) : x\in {\mathbb R},\ t>0\},$ equipped with the initial conditions $u(x,0)=f(x)\in C^{2}({\mathbb R})$ and $u_{t}(x,0)=g(x)\in C^{1}({\mathbb R}),$ is given by the d'Alambert formula
$$
u(x,t)=\frac{1}{2}\bigl[ f(x-at) +f(x+at) \bigr]+\frac{1}{2a}\int^{x+at}_{x-at}g(s)\, ds,\quad x\in {\mathbb R}, \ t>0.
$$  
Let us suppose that the function $x\mapsto (f(x),g^{[1]}(x)),$ $x\in {\mathbb R}$
is $c$-almost periodic, where 
$
g^{[1]}(\cdot) \equiv \int^{\cdot}_{0}g(s)\, ds.$  Then the solution 
$u(x,t)$ can be extended to the whole real line in the time variable and this solution is $c$-almost periodic in $(x,t)\in {\mathbb R}^{2}.$ To verify this, fix a positive real number $\epsilon>0.$ Then there exists a finite real number $l>0$ such that any subinterval $I$ of ${\mathbb R}$ of length $l$ contains a point $\tau \in I$ such that 
\begin{align}\label{capslock}
|f(x+\tau)-cf(x)|+\Bigl|g^{[1]}(x+\tau)-cg^{[1]}(x)\Bigr|<\epsilon,\quad x\in {\mathbb R}.
\end{align}
Furthermore, we have ($x,\ t,\ \tau_{1},\ \tau_{2}\in {\mathbb R}$):
\begin{align}\label{jarakb}
\begin{split}
\Bigl|u\bigl(x&+\tau_{1},t+\tau_{2}\bigr)-cu(x,t)\Bigr|
\\& \leq\frac{1}{2}\Bigl| f\bigl( (x-at)+(\tau_{1}-a\tau_{2}) \bigr)-cf(x-at)\Bigr|
\\&+ \frac{1}{2}\Bigl| f\bigl( (x+at)+(\tau_{1}+a\tau_{2}) \bigr)-cf([x+at+(\tau_{1}+a\tau_{2}) ]-(\tau_{1}+a\tau_{2}) )\Bigr|
\\& +\frac{1}{2a}\Bigl|g^{[1]}\bigl( (x-at)+(\tau_{1}-a\tau_{2}) \bigr)-cg^{[1]}(x-at)\Bigr|
\\& +\frac{1}{2a}\Bigl| g^{[1]}\bigl( (x+at)-(\tau_{1}-a\tau_{2}) \bigr)-cg^{[1]}(x+at)\Bigr|.
\end{split}
\end{align}
Let $(t_{1},t_{2})\in {\mathbb R}^{2}.$ Then the interval $[-t_{1}-at_{2}-(l/2),-t_{1}-at_{2}+(l/2)]$ contains the point $\tau'$ and 
the interval $[t_{1}-at_{2}-(l/2),t_{1}-at_{2}+(l/2)]$ contains the point $\tau''$ such that the equation \eqref{capslock} holds with the number $\tau$ replaced therein with any of the numbers $\tau',\ \tau''.$ Setting $\tau_{1}:=(\tau''-\tau')/2$
and $\tau_{2}:=(-\tau_{1}-\tau_{2})/2a,$ it can be easily shown that $|\tau_{1}-t_{1}|\leq l/2$ and $|\tau_{2}-t_{2}|\leq l/2a$, so that the final conclusion simply follows from the corresponding definition and \eqref{jarakb}. 
\end{example}

The notion of $(\omega,c)$-periodicity and various generalizations of this concept have recently been introduced and investigated by E. Alvarez, A. G\'omez, M. Pinto \cite{alvarez1} and
E. Alvarez, S. Castillo, M. Pinto \cite{alvarez2}-\cite{alvarez3}. In \cite{c1}, our joint paper with M. T. Khalladi, A. Rahmani, M. Pinto and D. Velinov,
we have recently introduced and analyzed the classes of
$c$-almost periodic functions, $c$-uniformly recurrent functions, semi-$c$-periodic functions 
and their Stepanov generalizations, where $c\in {\mathbb C}$ and $|c|=1.$ On the other hand, in \cite{marko-manuel-ap}, we have recently analyzed various notions of multi-dimensional almost periodic type functions. The main aim of this paper is to continue the research studies \cite{marko-manuel-ap} and \cite{c1} by investigating various notions of multi-dimensional $c$-almost periodic type functions and related applications, where $c\in {\mathbb C} \setminus \{0\}.$ For simplicity, we will not consider the corresponding Stepanov classes here (see \cite{nova-man} for more details).

Before going any further, the author would like to express his sincere thanks to Prof. A. Ch\'avez, M. T. Khalladi, M. Pinto, A. Rahmani and D. Velinov for many stimulating discussions during this investigation. The research articles \cite{marko-manuel-ap} and \cite{c1}, among many others, are written in a collaboration with these mathematicians.\vspace{0.2cm}

\noindent {\bf Notation and terminology.}
We assume henceforth that $(X,\| \cdot \|)$, $(Y, \|\cdot\|_Y)$ and $(Z, \|\cdot\|_Z)$ are complex Banach spaces and $n\in {\mathbb N}$;  usually, ${\mathcal B}$ denotes the collection of all bounded subsets of $X$ or all compact subsets
of $X.$ Set ${\mathcal B}_{X}:=\{y\in X : (\exists B\in {\mathcal B})\, y\in B\}.$ 
We will always assume henceforth that ${\mathcal B}_{X}
=X,$ i.e., 
that
for each $x\in X$ there exists $B\in {\mathcal B}$ such that $x\in B.$ By
$L(X,Y)$ we denote the Banach algebra of all bounded linear operators from $X$ into
$Y;$ $L(X,X)\equiv L(X)$. By $B^{\circ}$ and $\partial B$ we denote the interior and the boundary of a subset $B$ of a topological space $X$, respectively.

The symbol $C(I: X)$ stands for the space of all $X$-valued
continuous functions defined  on the domain $I$. By $C_{b}(I: X)$ (respectively, $BUC(I: X)$) we denote the subspace of $C(I: X)$ consisting of all bounded (respectively, all bounded uniformly continuous functions). Both $C_{b}(I: X)$ and $BUC(I: X)$ are Banach spaces with the sup-norm $\| \cdot\|_{\infty}$. This also holds for the space $C_{0}(I : X)$ consisting of all continuous functions $f : I \rightarrow X$ such that $\lim_{|t|\rightarrow +\infty}f(t)=0.$ 
 If ${\bf t_{0}}\in {\mathbb R}^{n}$ and $\epsilon>0$, then we set $B({\bf t}_{0},\epsilon):=\{{\bf t } \in {\mathbb R}^{n} : |{\bf t}-{\bf t_{0}}| \leq \epsilon\},$
where $|\cdot|$ denotes the Euclidean norm in ${\mathbb R}^{n}.$ Set 
${\mathbb N}_{n}:=\{1,\cdot \cdot \cdot, n\}$ and $S_{1}:=\{ z\in {\mathbb C} \, ; \, |z|=1\}.$ If any component of the tuple ${\bf t}=(t_{1},t_{2},\cdot \cdot \cdot, t_{n}) \in {\mathbb R}^{n}$ is strictly positive, then we simply write ${\bf t}> {\bf 0}.$ \vspace{0.2cm}

Now we will briefly explain the organization and main ideas of this paper. In Subsection \ref{stavisub}, we recall the basic definitions and results about almost periodic functions in ${\mathbb R}^{n}$. If $\emptyset  \neq I \subseteq {\mathbb R}^{n},$ $I +I \subseteq I$ and $F : I \times X \rightarrow Y$ is a continuous function, then 
the notions of Bohr $({\mathcal B},c)$-almost periodicity and $({\mathcal B},c)$-uniform recurrence for $F(\cdot;\cdot)$ are introduced in Definition \ref{nafaks1234567890}.
If the region $I$ satisfies certain conditions,
$F : I \times X \rightarrow Y$ is Bohr $({\mathcal B},c)$-almost periodic and ${\mathcal B}$ is any family of compact subsets of $X,$ then some sufficient conditions ensuring that 
for each set $B\in {\mathcal B}$ we have 
that the set $\{ F({\bf t}; x) : {\bf t} \in I,\ x\in B\}$ is relatively compact in $Y$
are given in Proposition \ref{bounded-pazice} (see also Proposition \ref{superstebag}, where we analyze 
the compositions of Bohr $({\mathcal B},c)$-almost periodic/$({\mathcal B},c)$-uniformly recurrent functions with uniformly continuous functions $\phi : Y \rightarrow Z$). 

The notion introduced in Definition \ref{nafaks1234567890} is reexamined and extended in Definition \ref{nafaks123456789012345}, where we introduce the notions of 
Bohr $({\mathcal B},I',c)$-almost periodicity and $({\mathcal B},I',c)$-uniform recurrence ($\emptyset  \neq I ' \subseteq I\subseteq {\mathbb R}^{n}$). Example \ref{dva naiks}, although very simple and elaborate, shows that 
the statement of \cite[Proposition 2.6]{c1} fails to be true for multi-dimensional $({\mathcal B},I',c)$-uniformly recurrent functions, in general. An important 
extension of \cite[Proposition 2.17]{c1} is proved in 
Proposition \ref{mentalnorazgib}, where condition 
$I +I' = I$ is crucial for proving the fact that we always have $c=\pm 1$ provided the existence of a $({\mathcal B},I',c)$-uniformly recurrent non-zero function $F : I \rightarrow {\mathbb R}$   (if $F({\bf t})\geq 0$ for all ${\bf t} \in I,$ then $c=1$); see also Example \ref{mammaru}. Proposition \ref{mentalnorazgib} is later employed in the proof of Proposition \ref{prcko-instrukt}, where it is shown that, if the function $F : I \times X \rightarrow Y$ is Bohr $({\mathcal B},I',c)$-almost periodic ($({\mathcal B},I',c)$-uniformly recurrent), $I+I'=I$ and $F(\cdot; \cdot)\neq 0,$ then $|c|=1.$ 

The first example of a multi-dimensional almost anti-recurrent function $F :{\mathbb R}^{n}\rightarrow {\mathbb R}$ ($c=-1$) which is not almost periodic is presented in Example \ref{rajkomilice}(iii)-(b).
After that, in Proposition \ref{jugosi}, we transfer the statement of
\cite[Proposition 2.9]{c1} for multi-dimensional Bohr $({\mathcal B}, c)$-almost periodic type functions
(see also Corollary \ref{jugosi1} and 
Proposition \ref{voliosam}
for similar results). The convolution invariance of Bohr $({\mathcal B},c)$-almost periodic type functions, invariance of Bohr $c$-almost periodicity and 
composition theorem for Bohr $({\mathcal B},c)$-almost periodic type functions
are investigated in Proposition \ref{convdiaggacece}, Proposition \ref{krucijaceq} and Theorem \ref{episkop-jovan}, respectively.
The main structural profilations of ${\mathbb D}$-asymptotically $c$-almost periodic type functions are given in Subsection \ref{gade-negadece}. In this subsection, we state and prove our main results, Theorem \ref{bounded-paziemceq} (in which we analyze certain relations between the classes of $I$-asymptotically Bohr $c$-almost periodic functions of type $1$ and $I$-asymptotically Bohr $c$-almost periodic functions)
and Theorem \ref{lenny-jassonce} (in which we analyze the extensions of Bohr $(I',c)$-almost periodic functions and $(I',c)$-uniformly recurrent functions). The final section of paper is reserved for applications of our abstract theoretical results. 

\subsection{Almost periodic functions on ${\mathbb R}^{n}$}\label{stavisub}

Suppose that $F : {\mathbb R}^{n} \rightarrow X$ is a continuous function. Let us recall that $F(\cdot)$ is said to be almost periodic if and only if for each $\epsilon>0$
there exists $l>0$ such that for each ${\bf t}_{0} \in {\mathbb R}^{n}$ there exists ${\bf \tau} \in B({\bf t}_{0},l)$ such that
\begin{align*}
\bigl\|F({\bf t}+{\bf \tau})-F({\bf t})\bigr\| \leq \epsilon,\quad {\bf t}\in {\mathbb R}^{n}.
\end{align*}
This is equivalent to saying that for any sequence $({\bf b}_n)$ in ${\mathbb R}^{n}$ there exists a subsequence $({\bf a}_{n})$ of $({\bf b}_n)$
such that $(F(\cdot+{\bf a}_{n}))$ converges in $C_{b}({\mathbb R}^{n}: X).$ 
The vector space of all almost periodic functions $F : {\mathbb R}^{n} \rightarrow X$ is denoted by $AP({\mathbb R}^{n} : X).$ Any almost periodic function $F : {\mathbb R}^{n} \rightarrow X$ is bounded and $AP({\mathbb R}^{n} : X)$ is the Banach space equipped with the sup-norm.

Any trigonometric polynomial in ${\mathbb R}^{n}$ is almost periodic and a continuous function $F(\cdot)$ is almost periodic if and only if there exists a sequence of trigonometric polynomials in ${\mathbb R}^{n}$ which converges uniformly to $F(\cdot);$
let us recall that a trigonometric polynomial in ${\mathbb R}^{n}$ is any linear combination of functions like ${\bf t}\mapsto e^{i\langle
 {\bf \lambda}, {\bf t} \rangle},$ ${\bf t}\in{\mathbb R}^{n},$ where ${\bf \lambda}\in {\mathbb R}^{n}$ and $\langle \cdot, \cdot \rangle$ denotes the usual inner product in ${\mathbb R}^{n}.$ 
Any almost periodic function $F : {\mathbb R}^{n} \rightarrow X$ is also uniformly continuous, the mean value
$$
M(F):=\lim_{T\rightarrow +\infty}\frac{1}{(2T)^{n}}\int_{{\bf s}+K_{T}}F({\bf t})\, d{\bf t}
$$
exists and it does not depend on $s\in {\mathbb R}^{n};$ here,  
$K_{T}:=\{ {\bf t}=(t_{1},t_{2},\cdot \cdot \cdot,t_{n}) \in {\mathbb R}^{n} :  |t_{i}|\leq T\mbox{ for }1\leq i\leq n\}.$ We define the Bohr-Fourier coefficient $F_{\lambda}\in X$ by \index{Bohr-Fourier coefficient}
$$
F_{\lambda}:=M\Bigl(e^{-i\langle \lambda, {\bf \cdot}\rangle }F(\cdot)\Bigr),\quad \lambda \in {\mathbb R}^{n},
$$
and the Bohr spectrum $
\sigma(F)$ of $F$ by\index{Bohr spectrum}
$$
\sigma(F):=\bigl\{ \lambda \in {\mathbb R}^{n} : F_{\lambda}\neq 0 \bigr\}.
$$
It is well known that $
\sigma(F)$
is at most a countable set. By $AP_{\Lambda}({\mathbb R}^{n} : X)$  
we denote the set consisting of all almost periodic functions $F : {\mathbb R}^{n} \rightarrow X$ 
such that $\sigma(F) \subseteq \Lambda.$ As is well known, for every almost periodic function $F \in AP_{\Lambda}({\mathbb R}^{n} : X),$
we can always find a sequence $(P_{k})$ of trigonometric polynomials in ${\mathbb R}^{n}$ which uniformly converges to $F(\cdot)$ on ${\mathbb R}^{n}$
and satisfies that $\sigma(P_{k}) \subseteq \Lambda$ for all $k\in {\mathbb N};$ see e.g., \cite[Chapter 1, Section 2.3]{pankov}.

\section{Bohr $({\mathcal B},c)$-almost periodic type functions}\label{maremareceacea}

The main aim of this section is to analyze Bohr $({\mathcal B},c)$-almost periodic type functions depending of several real variables, where
${\mathcal B}$ denotes a non-empty collection of non-empty subsets of $X$ and $c\in {\mathbb C} \setminus \{0\}.$ 
We will consider the following notion, which can be also analyzed on general topological (semi-)groups; see the references quoted in \cite{marko-manuel-ap} for more details concerning this problematic: 

\begin{defn}\label{nafaks1234567890}
Suppose that $\emptyset  \neq I \subseteq {\mathbb R}^{n},$ $F : I \times X \rightarrow Y$ is a continuous function and $I +I \subseteq I.$ Then we say that:
\begin{itemize}
\item[(i)]\index{function!Bohr $({\mathcal B},c)$-almost periodic}
$F(\cdot;\cdot)$ is Bohr $({\mathcal B},c)$-almost periodic if and only if for every $B\in {\mathcal B}$ and $\epsilon>0$
there exists $l>0$ such that for each ${\bf t}_{0} \in I$ there exists ${\bf \tau} \in B({\bf t}_{0},l) \cap I$ such that
\begin{align*}
\bigl\|F({\bf t}+{\bf \tau};x)-cF({\bf t};x)\bigr\|_{Y} \leq \epsilon,\quad {\bf t}\in I,\ x\in B.
\end{align*}
\item[(ii)] \index{function!$({\mathcal B},c)$-uniformly recurrent}
$F(\cdot;\cdot)$ is $({\mathcal B},c)$-uniformly recurrent if and only if for every $B\in {\mathcal B}$ 
there exists a sequence $({\bf \tau}_{k})$ in $I$ such that $\lim_{k\rightarrow +\infty} |{\bf \tau}_{k}|=+\infty$ and
\begin{align*}
\lim_{k\rightarrow +\infty}\sup_{{\bf t}\in I;x\in B} \bigl\|F({\bf t}+{\bf \tau}_{k};x)-cF({\bf t};x)\bigr\|_{Y} =0.
\end{align*}
\end{itemize}
If $X\in {\mathcal B},$ then it is also said that $F(\cdot;\cdot)$ is Bohr $c$-almost periodic ($c$-uniformly recurrent); if $c=1,$ then we also say that $F(\cdot;\cdot)$ is Bohr ${\mathcal B}$-almost periodic (${\mathcal B}$-uniformly recurrent) [Bohr almost periodic (uniformly recurrent)].
\end{defn}

Unless stated otherwise, we will assume that $\emptyset  \neq I \subseteq {\mathbb R}^{n}$ henceforth.
It is clear that any Bohr ($({\mathcal B},c)$-)almost periodic function is ($({\mathcal B},c)$-)uniformly recurrent; in general, the converse statement does not hold (\cite{nova-man}). In \cite[Proposition 2.2]{c1}, we have proved that any Bohr almost periodic function $f: I \rightarrow Y$ is bounded, provided that $I=[0,\infty)$ or $I={\mathbb R}.$ In the multi-dimensional case, the things become more complicated and the best we can do is to prove the following extension of the above-mentioned result following the method proposed in the proof of \cite[Proposition 2.16]{marko-manuel-ap}, which is applicable in the case that $I=[0,\infty)^{n}$ or $I={\mathbb R}^{n}:$

\begin{prop}\label{bounded-pazice}
Suppose that $\emptyset  \neq I \subseteq {\mathbb R}^{n},$ $I +I \subseteq I,$ $I$ is closed,
$F : I \times X \rightarrow Y$ is Bohr $({\mathcal B},c)$-almost periodic and ${\mathcal B}$ is any family of compact subsets of $X.$ If
\begin{align*}
\notag
(\forall l>0) \, (\exists {\bf t_{0}}\in I)\, (\exists k>0) &\, (\forall {\bf t} \in I)(\exists {\bf t_{0}'}\in I)\,
\\& (\forall {\bf t_{0}''}\in B({\bf t_{0}'},l) \cap I)\, {\bf t}- {\bf t_{0}''} \in B({\bf t_{0}},kl) \cap I,
\end{align*}
then for each $B\in {\mathcal B}$ we have 
that the set $\{ F({\bf t}; x) : {\bf t} \in I,\ x\in B\}$ is relatively compact in $Y;$
in particular,
$\sup_{{\bf t}\in I;x\in B}\|F({\bf t}; x)\|_{Y}<\infty.$
\end{prop}

We continue by providing the following illustrative example:

\begin{example}\label{pajos-langos} (see also \cite[Example 2.15]{c1})
Suppose that $\varphi \in (-\pi,\pi] \setminus \{0\},$ $\theta \in (-\pi,\pi],$ $\mu \in {\mathbb R}^{n} \setminus \{0\}$  and $c=e^{i\theta}.$ Then the trigonometric polynomial
${\bf t} \rightarrow e^{i\langle \mu, {\bf t}\rangle},$ ${\bf t}\in {\mathbb R}^{n}$ is $c$-almost periodic. Towards see this, set $S:=\{j\in {\mathbb N}_{n} : \mu_{j}\neq 0\}$ and $l:=\max\{2\pi |\mu_{j}|^{-1} : j\in S\}.$ Let $\epsilon>0$ be fixed. Then we have (${\bf t} \in {\mathbb R}^{n};$ $\tau \in {\mathbb R}^{n}$):
\begin{align*}
\Bigl|  & e^{i\langle \mu, {\bf t}+\tau\rangle}-e^{i\theta} e^{i\langle \mu, {\bf t}\rangle}\Bigr|
\\& =\Bigl| e^{i[\mu_{1}\tau_{1}+\mu_{2}\tau_{2}+\cdot \cdot \cdot +\mu_{n}\tau_{n}-\theta]}-1 \Bigr| =2\Biggl|\sin \Bigl( \frac{\mu_{1}\tau_{1}+\mu_{2}\tau_{2}+\cdot \cdot \cdot +\mu_{n}\tau_{n}-\theta}{2}\Bigr)\Biggr|,
\end{align*}
and therefore
\begin{align*}
\Bigl|  &e^{i\langle \mu, {\bf t}+\tau\rangle}-e^{i\theta} e^{i\langle \mu, {\bf t}\rangle}\Bigr|\leq \epsilon,\ {\bf t} \in {\mathbb R}^{n} \mbox{ if and only if there exists }k\in {\mathbb Z}\mbox{ such that }
\\& \mu_{1}\tau_{1}+\mu_{2}\tau_{2}+\cdot \cdot \cdot +\mu_{n}\tau_{n}-\theta \in \Bigl[ -\arcsin (\epsilon/2) +k\pi, \arcsin (\epsilon/2) +k\pi\Bigr].
\end{align*}
In particular, if there exists $k\in {\mathbb Z}$ such that $\mu_{1}\tau_{1}+\mu_{2}\tau_{2}+\cdot \cdot \cdot +\mu_{n}\tau_{n} =k\pi +\theta, $ then we have $| e^{i\langle \mu, {\bf t}+\tau\rangle}-e^{i\theta} e^{i\langle \mu, {\bf t}\rangle}|\leq \epsilon,$ $\ {\bf t} \in {\mathbb R}^{n}$. But, we can simply prove that for each ${\bf t}_{0} \in {\mathbb R}^{n}$ there exists a point $\tau \in B({\bf t}_{0},l)$ such that $\mu_{1}\tau_{1}+\mu_{2}\tau_{2}+\cdot \cdot \cdot +\mu_{n}\tau_{n} =k\pi +\theta$ for some $k\in {\mathbb Z},$ which simply implies the required.
\end{example}

Using a slight modification of the proof of \cite[Property 4, p. 3]{18}, we may conclude the following:

\begin{prop}\label{superstebag}
Suppose that $F : I \times X \rightarrow Y$ is Bohr $({\mathcal B},c)$-almost periodic/$({\mathcal B},c)$-uniformly recurrent, and $\phi : Y \rightarrow Z$ is uniformly continuous on $\overline{R(F)}$ and satisfies that $\phi (cy)=c\phi(y)$ for all $y\in Y.$
Then $\phi \circ F : I \times X \rightarrow Z$ is Bohr $({\mathcal B},c)$-almost periodic/$({\mathcal B},c)$-uniformly recurrent.
\end{prop}

The conclusions clarified in the next illustrative example follow from the arguments similar to those employed in \cite[Example 2.13]{marko-manuel-ap}:

\begin{example}\label{pwerqwer}
\begin{itemize}
\item[(i)] Suppose that $F_{j} : X \rightarrow Y$ is a continuous function, for each $B\in {\mathcal B}$ we have $\sup_{x\in B}\| F_{j}(x) \|_{Y}<\infty$ 
and the complex-valued mapping $t\mapsto (\int_{0}^{t}f_{1}(s)\, ds,\cdot \cdot \cdot, \int_{0}^{t}f_{n}(s)\, ds),$  $t\geq 0$ is $c$-almost periodic ($1\leq j \leq n$). Set
\begin{align*}
F\bigl(t_{1},\cdot \cdot \cdot,t_{n+1}; x\bigr):=\sum_{j=1}^{n}\int_{t_{j}}^{t_{j+1}}f_{j}(s)\, ds \cdot F_{j}(x)\ \mbox{ for all }x\in X \mbox{ and } t_{j}\geq 0,\ 1\leq j\leq n.
\end{align*}
Then the mapping $F: [0,\infty)^{n+1} \times X \rightarrow X$ is Bohr $({\mathcal B},c)$-almost periodic. 
\item[(ii)] Suppose that $F : X \rightarrow Y$ is a continuous function, for each $B\in {\mathcal B}$ we have $\sup_{x\in B}\| F(x) \|_{Y}<\infty$ and the complex-valued mapping $t\mapsto f_{j}(t),$ $t\geq 0$ is $c$-almost periodic, resp. bounded and $c$-uniformly recurrent ($1\leq j \leq n$). Set
\begin{align*}
F\bigl(t_{1},\cdot \cdot \cdot,t_{n}; x\bigr):=\prod_{j=1}^{n}f_{j}\bigl(t_{j}\bigr)\cdot F(x)\ \mbox{ for all }x\in X \mbox{ and } t_{j}\geq 0,\ 1\leq j\leq n.
\end{align*}
Then the mapping $F: [0,\infty)^{n} \times X \rightarrow X$ is Bohr $({\mathcal B},c)$-almost periodic, resp. $({\mathcal B},c)$-uniformly recurrent.
\item[(iii)] Suppose that $G : [0,\infty)^{n} \rightarrow {\mathbb C}$ is $c$-almost periodic, resp. bounded and $c$-uniformly recurrent, $F: [0,\infty) \times X \rightarrow Y$ is Bohr ${\mathcal B}$-almost periodic, resp. ${\mathcal B}$-uniformly recurrent, and for each set $B\in {\mathcal B}$ we have $\sup_{t\geq 0; x\in B}\| F(t;x) \|_{Y}<\infty.$ Set
\begin{align*}
F\bigl(t_{1},\cdot \cdot \cdot,t_{n+1}; x\bigr)&:=G\bigl(t_{1},\cdot \cdot \cdot, t_{n}\bigr)\cdot F\bigl(t_{n+1};x\bigr)\\& \mbox{ for all }x\in X \mbox{ and } t_{j}\geq 0,\ 1\leq j\leq n+1.
\end{align*}
Then the mapping $F: [0,\infty)^{n+1} \times X \rightarrow Y$ is Bohr $({\mathcal B},c)$-almost periodic, resp. $({\mathcal B},c)$-uniformly recurrent (see also \cite[Proposition 2.19, Example 2.22]{marko-manuel-ap}, which can be modified in a similar fashion).
\end{itemize}
\end{example}

The notion introduced in Definition \ref{nafaks1234567890} can be extended as follows:

\begin{defn}\label{nafaks123456789012345}
Suppose that $\emptyset  \neq I'\subseteq I \subseteq {\mathbb R}^{n},$ $F : I \times X \rightarrow Y$ is a continuous function and $I +I' \subseteq I.$ Then we say that:
\begin{itemize}
\item[(i)]\index{function!Bohr $({\mathcal B},I',c)$-almost periodic}
$F(\cdot;\cdot)$ is Bohr $({\mathcal B},I',c)$-almost periodic if and only if for every $B\in {\mathcal B}$ and $\epsilon>0$
there exists $l>0$ such that for each ${\bf t}_{0} \in I'$ there exists ${\bf \tau} \in B({\bf t}_{0},l) \cap I'$ such that
\begin{align}\label{ujshe}
\bigl\|F({\bf t}+{\bf \tau};x)-c F({\bf t};x)\bigr\|_{Y} \leq \epsilon,\quad {\bf t}\in I,\ x\in B.
\end{align}
\item[(ii)] \index{function!$({\mathcal B},I',c)$-uniformly recurrent}
$F(\cdot;\cdot)$ is $({\mathcal B},I',c)$-uniformly recurrent if and only if for every $B\in {\mathcal B}$ 
there exists a sequence $({\bf \tau}_{k})$ in $I'$ such that $\lim_{k\rightarrow +\infty} |{\bf \tau}_{k}|=+\infty$ and
\begin{align}\label{ujshe-rec}
\lim_{k\rightarrow +\infty}\sup_{{\bf t}\in I;x\in B} \bigl\|F({\bf t}+{\bf \tau}_{k};x)-cF({\bf t};x)\bigr\|_{Y} =0.
\end{align}
\end{itemize}
If $X\in {\mathcal B},$ then it is also said that $F(\cdot;\cdot)$ is Bohr $(I',c)$-almost periodic ($(I',c)$-uniformly recurrent).
\end{defn}

\begin{rem}\label{era-cea}
\begin{itemize}
\item[(i)] Let $|c|=1$ and $F : {\mathbb R} \rightarrow Y$ be a continuous function. Then
$F(\cdot)$ is $c$-almost periodic ($c$-uniformly recurrent) in the sense of \cite{c1} if and only if $F(\cdot)$ is
Bohr $((0,\infty),c)$-almost periodic ($((0,\infty),c)$-uniformly recurrent) in the sense of Definition \ref{nafaks123456789012345}. Albeit we will not consider here the general question concerning the existence of larger sets $I''\supseteq I'$ for which a given a Bohr $({\mathcal B},I',c)$-almost periodic function $F(\cdot;\cdot)$ is also $({\mathcal B},I'',c)$-almost periodic (the only exception is the proof of Theorem \ref{lenny-jassonce}), we would like to note that any Bohr $((0,\infty),c)$-almost periodic function is already Bohr $({\mathbb R},c)$-almost periodic.
This is clear if $\arg(c) /\pi \notin {\mathbb Q}$ since we can apply then \cite[Proposition 2.11(i)]{c1} in order to see that the function $F(\cdot)$ is also Bohr $((0,\infty),c^{-1})$-almost periodic and therefore, given $\epsilon>0$ in advance, we can collect all positive $(\epsilon,c)$-periods of function $F(\cdot)$ and all negative values of all positive  $(\epsilon,c^{-1})$-periods of function $F(\cdot)$ (with the meaning clear), obtaining thus a relatively dense set in ${\mathbb R}$ consisting solely of $(\epsilon,c)$-periods of $F(\cdot).$ The situation is similar if
$\arg(c) /\pi \in {\mathbb Q}$ because then there exists $m\in {\mathbb N}$ such that $c^{m+1}=1$ so that $c^{m}=c^{-1}$ and we can collect 
 all positive $(\epsilon,c)$-periods of function $F(\cdot)$ and all
negatives of all positive $(\epsilon/m,c)$-periods of function $F(\cdot)$ in order to obtain a relatively dense set in ${\mathbb R}$ consisting solely of $(\epsilon,c)$-periods of $F(\cdot);$ observe here only that the assumption $\|F(t+\tau)-cF(t)\|\leq \epsilon$ for all $t\in {\mathbb R}$ and some $\tau \in {\mathbb R}$ implies 
\begin{align}
&\label{valevo} \bigl\|  F(t+m\tau)-c^{m}F(t)\bigr\| 
\\\notag & \leq \bigl\| F(t+m\tau)-cF(t+(m-1)\tau)\bigr\| +|c|\bigl\| F(t+(m-1)\tau)-cF(t+(m-2)\tau)\bigr\| 
\\\notag &+\cdot \cdot \cdot +|c|^{m-2}\bigl\| F(t+2\tau)-cF(t+\tau)\bigr\| +|c|^{m-1}\bigl\| F(t+\tau)-cF(t)\bigr\| 
  \leq m\epsilon,\ t\in {\mathbb R}. 
\end{align}
\item[(ii)] Condition $\emptyset  \neq I'\subseteq I $ is a bit unnecessary and intended for considerations of regions $I$ for which $0\in I;$ more precisely, the assumption $I+I'\subseteq I$ is mandatory and implies that for each ${\bf t}_{0}\in I$ we have $I' \subseteq I-{\bf t}_{0}$ (take, for example $I=[1,\infty)$ and $I'=[0,\infty);$ 
 then we do not have $I'\subseteq I$ but the notion introduced in Definition \ref{nafaks123456789012345} is meaningful).
\item[(iii)] The main structural properties of functions introduced in Definition \ref{nafaks1234567890} and Definition \ref{nafaks123456789012345}, clarified in \cite[Proposition 2.28]{c1} and \cite[Theorem 2.13]{c1}, continue to hold with appropriate modifications. For example, the introduced spaces of functions are translation invariant, in a certain sense, with respect to the both variables.
\end{itemize}
\end{rem}

Clearly, 
the notion from Definition \ref{nafaks1234567890} is recovered by plugging $I'=I$ and any 
$({\mathcal B},I',c)$-uniformly recurrent function is $({\mathcal B},I,c)$-uniformly recurrent
provided that $I+I\subseteq I.$ 
Concerning the statement of \cite[Proposition 2.6]{c1}, we would like to present first the following instructive example:  

\begin{example}\label{dva naiks}
Suppose that $I:=\{ (x,y)\in {\mathbb R}^{2} : x+y\geq 0\}$ ($I:=\{ (x,y)\in {\mathbb R}^{2} : x+y\geq 0\}$) and $I':=\{ (x,y)\in {\mathbb R}^{2} : x+y=1\}$ ($I':=\{ (x,y)\in {\mathbb R}^{2} : x+y=-1\}$). Set $F(x,y):=2^{
-x-y} ,$ $(x,y)\in I.$ Then $I+I'\subseteq I+I =I$ and for every $(a,b)\in I'$ we have $F((x,y)+(a,b))=2^{-1}F(x,y),$ $(x,y)\in I$ ($F((x,y)+(a,b))=2F(x,y),$ $(x,y)\in I$), so that $F(\cdot, \cdot)$ is both Bohr $(I',2^{-1})$-almost periodic and $2^{-1}$-uniformly recurrent (Bohr $(I',2)$-almost periodic and $2$-uniformly recurrent) but not identically equal to zero. 
\end{example}

Further on, the statement of \cite[Proposition 2.7]{c1} can be simply reformulated in our new framework and, if the function $F(\cdot;\cdot)$ is Bohr $({\mathcal B},I',c)$-almost periodic ($({\mathcal B},I',c)$-uniformly recurrent), then the function $\|F(\cdot ;\cdot)\|_{Y}$ is  Bohr $({\mathcal B},I',|c|)$-almost periodic ($({\mathcal B},I',|c|)$-uniformly recurrent). The following fact should be also clarified: If the function $F(\cdot;\cdot)$ is $({\mathcal B},I',c)$-uniformly recurrent, then for each $B\in {\mathcal B}$ we have
\begin{align}\label{tupak12345ce}
\sup_{{\bf t}\in I,x\in B}\bigl\|F({\bf t};x) \bigr\|_{Y}\leq |c|^{-1}\sup_{{\bf t}\in I,|{\bf t}|\geq a, {\bf t} \in I+I'}\bigl\|F({\bf t};x) \bigr\|_{Y},
\end{align}
and for each $x\in X$ the function $F(\cdot;x)$ is identically equal to zero provided that the function $F(\cdot;\cdot)$ is $({\mathcal B},I',c)$-uniformly recurrent and $\lim_{|{\bf t}|\rightarrow +\infty, {\bf t}\in I+I'}F({\bf t};x)=0.$ 

Now we are able to state and prove the following extension of \cite[Proposition 2.17]{c1}:

\begin{prop}\label{mentalnorazgib}
Suppose that $\emptyset  \neq I'\subseteq I \subseteq {\mathbb R}^{n}$ and $I +I' = I.$ If the function $F : I \rightarrow {\mathbb R}$ is $({\mathcal B},I',c)$-uniformly recurrent and $F\neq 0,$ then $c=\pm 1.$ Furthermore, if $F({\bf t})\geq 0$ for all ${\bf t} \in I,$ then $c=1.$
\end{prop}

\begin{proof}
Since we have assumed $I +I' = I$ and $F\neq 0,$ the equation \eqref{tupak12345ce} yields the existence of a finite real number 
$a>0$ and a sequence $({\bf t}_{k})$ in $I$ such that $|F({\bf t}_{k})|>a/2$ for all $k\in {\mathbb N}.$ Then the final conclusion follows by repeating verbatim the arguments contained in the proof of \cite[Proposition 2.17]{c1}.
\end{proof}

\begin{rem}\label{mammaru}
Suppose that $c=1/2$ in Example \ref{dva naiks}. Then the function $F(\cdot;\cdot)$ is real-valued so that the conclusion of Proposition \ref{mentalnorazgib} does not hold if the assumption $I+I'\neq I$ is neglected.
\end{rem}

The most important corollary of Proposition \ref{mentalnorazgib}, which extends the statement of \cite[Proposition 2.6]{c1}, is stated below:

\begin{cor}\label{prcko-instrukt}
Suppose that $\emptyset  \neq I'\subseteq I \subseteq {\mathbb R}^{n},$ $I+I'=I$ and $F : I \times X \rightarrow Y$ is Bohr $({\mathcal B},I',c)$-almost periodic ($({\mathcal B},I',c)$-uniformly recurrent). If $F(\cdot; \cdot)\neq 0,$ then $|c|=1.$
\end{cor}

\begin{proof}
By our assumption, there exist ${\bf t}_{0}\in I$ and $x\in X$ such that $F({\bf t}_{0};x)\neq 0.$ Further on, there exists $B\in {\mathcal B}$ such that $x\in B$ and this simply implies that the function $F_{x} : I \rightarrow Y$ is  Bohr $({\mathcal B},I',c)$-almost periodic ($({\mathcal B},I',c)$-uniformly recurrent) and not identically equal to zero. Therefore, the function $\|F_{x}(\cdot )\|_{Y}$ is Bohr $({\mathcal B},I',|c|)$-almost periodic ($({\mathcal B},I',|c|)$-uniformly recurrent) and not identically equal to zero. By Proposition \ref{mentalnorazgib}, we get that $|c|=1.$
\end{proof}

If $c = \pm 1,$ then we also say that the function $F(\cdot)$ is Bohr ${\mathcal B}$-almost (anti-)periodic (${\mathcal B}$-uniformly (anti-)recurrent)/Bohr $({\mathcal B},I')$-almost (anti-)periodic ($({\mathcal B},I')$-uniformly (anti-)recurrent). 
Let us recall that there is a great number of very simple examples showing that the notion of $({\mathcal B},I')$-almost periodicity is neither stronger nor weaker than the notion of $({\mathcal B},I)$-almost periodicity, provided that $I+I\subseteq I.$  

The conclusions established in the subsequent example follow similarly as in \cite[Example 2.15]{marko-manuel-ap}:  

\begin{example}\label{rajkomilice}
\begin{itemize}
\item[(i)]
Suppose that 
the complex-valued mapping $t\mapsto \int_{0}^{t}f_{j}(s)\, ds,$ $t\in {\mathbb R}$ is $c$-almost periodic, resp. bounded and $c$-uniformly recurrent ($1\leq j\leq n$). Set
\begin{align*}
F_{1}\bigl(t_{1},\cdot \cdot \cdot,t_{2n}\bigr):=\prod_{j=1}^{n}\int_{t_{j}}^{t_{j+n}}f_{j}(s)\, ds\ \mbox{ and } t_{j}\in {\mathbb R},\ 1\leq j\leq 2n.
\end{align*}
Then the mapping $F_{1}: {\mathbb R}^{2n} \rightarrow {\mathbb C}$ is Bohr $(I',c)$-almost periodic, resp. $(I',c)$-uniformly recurrent, where
$I'=\{({\bf \tau},{\bf \tau}) : {\bf \tau} \in {\mathbb R^{n}} \};$ furthermore, if the function
$t\mapsto (\int_{0}^{t}f_{1}(s)\, ds, \cdot \cdot \cdot, \int_{0}^{t}f_{n}(s)\, ds),$ $t\in {\mathbb R}$ is $c$-almost periodic, resp. bounded and $c$-uniformly recurrent, then the function
$F_{1}(\cdot)$ is Bohr $(I'',c)$-almost periodic, 
resp. $(I'',c)$-uniformly recurrent, where
$I''=\{(a,a,\cdot \cdot \cdot, a) \in {\mathbb R}^{2n} : a\in {\mathbb R}\}.$
\item[(ii)] Suppose that 
an $X$-valued mapping $t\mapsto \int_{0}^{t}f_{j}(s)\, ds,$ $t\in {\mathbb R}$ is $c$-almost periodic, resp. bounded and $c$-uniformly recurrent, as well as that a strongly continuous operator family $(T_{j}(t))_{t\in {\mathbb R}}\subseteq L(X,Y)$ is uniformly bounded ($1\leq j \leq n$). Set
\begin{align*}
F_{2}\bigl(t_{1},\cdot \cdot \cdot,t_{2n}\bigr)&:=\sum_{j=1}^{n}T_{j}(t_{j}-t_{j+n})\int_{t_{j}}^{t_{j+n}}f_{j}(s)\, ds\\&  \mbox{ and } t_{j}\in {\mathbb R},\ 1\leq j\leq 2n.
\end{align*}
Then the mapping $F_{2}: {\mathbb R}^{2n} \rightarrow {\mathbb C}$ is Bohr $(I',c)$-almost periodic, resp. $(I',c)$-uniformly recurrent, where $I'=\{({\bf \tau},{\bf \tau}) : {\bf \tau} \in {\mathbb R^{n}} \}$, but not generally Bohr $c$-almost periodic,
in the case of consideration of almost periodicity; furthermore, if the function $t\mapsto (\int_{0}^{t}f_{1}(s)\, ds, \cdot \cdot \cdot, \int_{0}^{t}f_{n}(s)\, ds),$ $t\in {\mathbb R}$ is $c$-almost periodic, resp. bounded and $c$-uniformly recurrent, then
the function $F_{2}(\cdot)$ is
Bohr $I''$-almost periodic, where
$I''=\{(a,a,\cdot \cdot \cdot, a) \in {\mathbb R}^{2n} : a\in {\mathbb R}\}.
$
\item[(iii)]
Suppose that $\emptyset  \neq I \subseteq {\mathbb R}^{n},$ 
$ I_{0}=[0,\infty)$ or $ I_{0}={\mathbb R},$ ${\bf a}=(a_{1},\cdot \cdot \cdot,a_{n}) \in {\mathbb R}^{n} \neq 0$ and the linear function
$
g({\bf t}):=a_{1}t_{1}+\cdot \cdot \cdot +a_{n}t_{n},$ ${\bf t}=(t_{1},\cdot \cdot \cdot,t_{n})\in I
$
maps surjectively the region $I$
onto $ I_{0}.$ Suppose, further, that $f: I_{0} \rightarrow X$ is a $c$-uniformly recurrent function as well as that a sequence $(\alpha_{k})$ in $I_{0}$ satisfies 
that $\lim_{k\rightarrow +\infty} |\alpha_{k}|=+\infty$ and 
$\lim_{k\rightarrow +\infty}\sup_{t\in I_{0}} \bigl\|f(t+\alpha_{k})-c f(t)\bigr\| =0.$ 
Define $I':=g^{-1}(\{\alpha_{k} : k\in {\mathbb N}\})$ and
$F : I \rightarrow X$ by
$
F({\bf t}):=f(g({\bf t})),$ ${\bf t}\in I.
$
Then $F(\cdot)$ is $(I',c)$-uniformly recurrent, and $F(\cdot )$ is not $c$-almost periodic provided that $f(\cdot)$ is not $c$-almost periodic (note that the conclusions established in \cite[Example 2.12]{marko-manuel-ap} cannot be reformulated for the $c$-uniform recurrence). 
We will provide two illustrative examples of the obtained conclusion as follows:
\begin{itemize}
\item[(a)] It is worth recalling that A. Haraux and P. Souplet have proved, in \cite[Theorem 1.1]{haraux}, that the
function $f: {\mathbb R}\rightarrow {\mathbb R},$ given by
\begin{align*}
f(t):=\sum_{n=1}^{\infty}\frac{1}{n}\sin^{2}\Bigl(\frac{t}{2^{n}} \Bigr)\, dt,\quad t\in {\mathbb R},
\end{align*}
is uniformly continuous, uniformly recurrent (the sequence\\ $(\tau_{k}\equiv 2^{k}\pi)_{k\in {\mathbb N}}$ can be chosen in definition of uniform recurrence) and unbounded; in \cite[Example 2.19(i)]{c1}, we have shown that $f(\cdot)$ is $c$-uniformly recurrent if and only if $c=1.$ Let ${\bf a}=(a_{1},\cdot \cdot \cdot,a_{n}) \in {\mathbb R}^{n} \setminus \{0\},$ let $I'=g^{-1}(\{2^{k}\pi : k\in {\mathbb N}\})$ and let
$F : {\mathbb R}^{n} \rightarrow {\mathbb R}$ be given by $F({\bf t}):=f(a_{1}t_{1}+\cdot \cdot \cdot +a_{n}t_{n}),$ ${\bf t}=(t_{1},\cdot \cdot \cdot,t_{n})\in {\mathbb R}^{n}.$
Then the function $F(\cdot)$ is uniformly continuous, unbounded, $I'$-uniformly recurrent and not almost periodic (\cite{marko-manuel-ap}); furthermore, an application of Proposition \ref{mentalnorazgib} shows that $F(\cdot)$ is $(I',c)$-uniformly recurrent if and only if $c=1.$
\item[(b)] In \cite[Example 2.20]{c1}, we have proved that the function $g: {\mathbb R}\rightarrow {\mathbb R},$ given by
\begin{align*}
f(t):=(\sin t) \cdot \sum_{n=1}^{\infty}\frac{1}{n}\sin^{2}\Bigl(\frac{t}{3^{n}} \Bigr),\quad t\in {\mathbb R},
\end{align*}
is $c$-uniformly recurrent if and only if $c=\pm 1$ (the sequence $(\tau_{k}\equiv 3^{k}\pi)_{k\in {\mathbb N}}$ can be chosen in definition of uniform anti-recurrence). 
Let ${\bf a} \in {\mathbb R}^{n} \setminus \{0\},$ let $I'=g^{-1}(\{3^{k}\pi : k\in {\mathbb N}\})$ and let
$F : {\mathbb R}^{n} \rightarrow {\mathbb R}$ be defined as in (a).
Then the function $F(\cdot)$ is uniformly continuous, unbounded, $I'$-uniformly anti-recurrent and not almost periodic; furthermore, an application of Proposition \ref{mentalnorazgib} shows that $F(\cdot)$ is $(I',c)$-uniformly recurrent if and only if $c=\pm 1.$
\end{itemize}
\end{itemize}
\end{example}

Set $lI':=\{l{\bf t} : {\bf t}\in I'\}$ for all $l\in {\mathbb N}.$ The following result extends \cite[Proposition 2.9]{c1} for $c$-almost periodic functions and $c$-uniformly recurrent functions:

\begin{prop}\label{jugosi}
Suppose that $l\in {\mathbb N},$ 
$\emptyset  \neq I'\subseteq I \subseteq {\mathbb R}^{n},$ $I +I' \subseteq I$ and $F : I \times X \rightarrow Y$ is 
Bohr $({\mathcal B},I',c)$-almost periodic ($({\mathcal B},I',c)$-uniformly recurrent). Then $ lI'\subseteq I,$
$I+lI'\subseteq I$
and
$F(\cdot;\cdot)$ is Bohr $({\mathcal B},lI',c^{l})$-almost periodic ($({\mathcal B},lI',c^{l})$-uniformly recurrent).
\end{prop}

\begin{proof}
Since $I'\subseteq I$ and $I+I'\subseteq I,$ we inductively get that $ jI'\subseteq I$ and
$I+jI'\subseteq I$ for all $j\in {\mathbb N}.$ Keeping this in mind, the proof simply follows from the corresponding definitions and the identity
(${\bf t}\in I,$ $\tau \in I'$):
\begin{align*}
F\bigl({\bf t}+l\tau\bigr)-c^{l}F({\bf t})
=\sum_{j=0}^{l-1}c^{j}\Bigl[ F\bigl({\bf t}+(l-j)\tau\bigr)- c F\bigl({\bf t}+(l-j-1)\tau\bigr)\Bigr].
\end{align*}
\end{proof}

Suppose now that:
\begin{align}\label{rasta123456}
p\in {\mathbb Z} \setminus \{0\},\ q\in {\mathbb N},\ (p,q)=1,\ |c|=1 \mbox{ and }\arg(c)=\pi p/q.
\end{align}
The most important corollary of Proposition \ref{jugosi}, which extends \cite[Corollary 2.10]{c1}, follows by plugging $l=q:$

\begin{cor}\label{jugosi1}
Suppose that \eqref{rasta123456} holds, 
$\emptyset  \neq I'\subseteq I \subseteq {\mathbb R}^{n},$ $I +I' \subseteq I$ and $F : I \times X \rightarrow Y$ is 
Bohr $({\mathcal B},I',c)$-almost periodic ($({\mathcal B},I',c)$-uniformly recurrent).
Then the following holds:
\begin{itemize}
\item[(i)] If $p$ is even, then $F(\cdot;\cdot)$ is Bohr $({\mathcal B},qI')$-almost periodic ($({\mathcal B},qI')$-uniformly recurrent).
\item[(ii)] If $p$ is odd, then $F(\cdot;\cdot)$ is Bohr $({\mathcal B},qI')$-almost anti-periodic ($({\mathcal B},qI')$-uniformly anti-recurrent).
\end{itemize}
\end{cor}

Similarly we can prove the following:

\begin{prop}\label{jugosi123}
Suppose that $|c|=1,$ $\arg(c)\in \pi {\mathbb Q}$, 
$\emptyset  \neq I'\subseteq I \subseteq {\mathbb R}^{n},$ $I +I' \subseteq I$ and $F : I \times X \rightarrow Y$ is 
Bohr $({\mathcal B},I',c)$-almost periodic ($({\mathcal B},I',c)$-uniformly recurrent). Define $C_{c}:=\{l\in {\mathbb N} : c^{l}=1\}$ and $C_{c,-1} :=\{l\in {\mathbb N} : c^{l}=-1\}.$
If $S$ is any finite non-empty subset of $C_{c},$ resp. $C_{c,-1},$ and $I_{S}':=\bigcup_{l\in S}lI',$ then $F(\cdot;\cdot)$ is Bohr $({\mathcal B},I'_{S})$-almost periodic ($({\mathcal B},I_{S}')$-uniformly recurrent), resp. Bohr $({\mathcal B},I_{S}')$-almost anti-periodic ($({\mathcal B},I_{S}')$-uniformly anti-recurrent).
\end{prop}

The subsequent result follows from the argumentation contained in the proof of \cite[Proposition 2.11(i)]{c1}:

\begin{prop}\label{voliosam}
Let $|c|=1$ and $\arg(c)/\pi \notin {\mathbb Q}.$ If $\emptyset  \neq I'\subseteq I \subseteq {\mathbb R}^{n},$ $I +I' \subseteq I,$
$lI'=I'$ for all $l\in {\mathbb N}$ and $F : I \times X \rightarrow Y$ is a  bounded, Bohr $({\mathcal B},I',c)$-almost periodic ($({\mathcal B},I',c)$-uniformly recurrent) function,
then the function $F(\cdot;\cdot)$ is Bohr $({\mathcal B},I',c)$-almost periodic ($({\mathcal B},I',c)$-uniformly recurrent) for all $c'\in S_{1}.$ 
\end{prop}

Concerning the convolution invariance of introduced spaces of Bohr $({\mathcal B},c)$-almost periodic type functions, we would like to state the following result:

\begin{prop}\label{convdiaggacece}
Suppose that $h\in L^{1}({\mathbb R}^{n}),$ $\emptyset  \neq I'\subseteq {\mathbb R}^{n}$ and the function $F(\cdot ; \cdot)$ is
Bohr $({\mathcal B},I',c)$-almost periodic ($({\mathcal B},I',c)$-uniformly recurrent). 
If
\begin{itemize}
\item[$(B)_{b}$:] For every $B\in {\mathcal B},$ there exists a finite real constant $c_{B}>0$ such that $\sup_{{\bf t}\in {\mathbb R}^{n}, x\in B}\|F({\bf t};x)\|_{Y}\leq c_{B}, $
\end{itemize}
then the function 
$$
(h\ast F)({\bf t};x):=\int_{{\mathbb R}^{n}}h(\sigma)F({\bf t}-\sigma;x)\, d\sigma,\quad {\bf t}\in {\mathbb R}^{n},\ x\in X
$$
is Bohr $({\mathcal B},I',c)$-almost periodic ($({\mathcal B},I',c)$-uniformly recurrent) and satisfies \emph{$(B)_{b}$}.
\end{prop}

\begin{proof}
Since $h\in L^{1}({\mathbb R}^{n}),$ the prescribed assumptions imply that
the function $(h\ast F)(\cdot ;\cdot)$ is well defined and satisfies $(B)_{b}$.
The continuity of function $(h\ast F)(\cdot ;\cdot)$ follows from the dominated convergence theorem, the continuity of the function $F(\cdot;\cdot)$ and condition $(B)_{b}$. Let $B\in {\mathcal B}$ and $\epsilon>0$ be fixed.
Then there exists $l>0$ such that for each ${\bf t}_{0} \in I'$ there exists ${\bf \tau} \in B({\bf t}_{0},l) \cap I'$ such that
\eqref{ujshe} holds with $I={\mathbb R}^{n}.$
Therefore,
\begin{align*}
\Bigl\| & (h\ast F)({\bf t}+\tau;x) -c\bigl( h\ast F)({\bf t}; x)\Bigr\|_{Y}
\\& \leq  \int_{{\mathbb R}^{n}}|h(\sigma)| \cdot \Bigl\|F({\bf t}+\tau-\sigma;x)-c
F({\bf t}-\sigma; x)\Bigr\|_{Y}\, d\sigma,
\end{align*}
for any ${\bf t}\in {\mathbb R}^{n}$ and $x\in B.$ This simply implies the required.
\end{proof}

The following result, which has recently been considered in \cite{marko-manuel-ap} in the case that $c=1,$ can be slightly extended for the Stepanov classes of $c$-almost periodic type functions (see the forthcoming monograph \cite{nova-man} for more details): 

\begin{prop}\label{krucijaceq}
Let $(R({\bf t}))_{{\bf t}> {\bf 0}}\subseteq L(X,Y)$ be a strongly continuous operator family such that
$\int_{(0,\infty)^{n}}\|R({\bf t} )\|\, d{\bf t}<\infty .$ If $f : {\mathbb R}^{n} \rightarrow X$ is $c$-almost periodic, then the function $F: {\mathbb R}^{n} \rightarrow Y,$ given by
\begin{align*}
F({\bf t}):=\int^{t_{1}}_{-\infty}\int^{t_{2}}_{-\infty}\cdot \cdot \cdot \int^{t_{n}}_{-\infty} R({\bf t}-{\bf s})f({\bf s})\, d{\bf s},\quad {\bf t}\in {\mathbb R}^{n},
\end{align*}
is well-defined and $c$-almost periodic.
\end{prop}

Suppose now that $|c|=1.$ Concerning the assertion of \cite[Theorem 2.24]{c1}, we will first observe that any almost periodic function $F\in AP_{{\mathbb R}^{n} \setminus \{0\}}({\mathbb R}^{n} : X)$ can be uniformly approximated by trigonometric polynomials whose frequencies belong to the set
${\mathbb R}^{n} \setminus \{0\}.$ If we denote by $AP_{c,0}({\mathbb R}^{n} : X)$ the linear span of all $c$-almost periodic functions $F : {\mathbb R}^{n} \rightarrow X$ and by ${\bf AP}_{c,0}({\mathbb R}^{n} : X)$ its closure in $AP({\mathbb R}^{n} : X),$
then it follows from the above and our conclusion established in Example \ref{pajos-langos} that $AP_{{\mathbb R}^{n} \setminus \{0\}}({\mathbb R}^{n} : X) \subseteq {\bf AP}_{c,0}({\mathbb R}^{n} : X).$ But, it is not clear how to prove or disprove the converse inclusion provided that arg$(c)\in \pi \cdot {\mathbb Q}.$

Before we move ourselves to the next subsection, we will state and prove a composition theorem for multi-dimensional Bohr $({\mathcal B},c)$-almost periodic type functions. Suppose that $F : I \times X \rightarrow Y$ and $G : I \times Y \rightarrow Z$ are given functions; then the multi-dimensional Nemytskii operator
$W : I  \times X \rightarrow Z$ is defined by
\begin{align}\label{skadar}
W({\bf t}; x):=G\bigl({\bf t} ; F({\bf t}; x)\bigr),\quad {\bf t} \in I,\ x\in X.
\end{align}
Set $R(F)\equiv \{ F({\bf t} ; x) : {\bf t} \in I, \ x\in X\}$ and suppose that there exists a finite real constant $L>0$ such that
\begin{align}\label{lajbaha}
\bigl\|G({\bf t};y)-G\bigl({\bf t};y'\bigr)\bigr\|_{Z} \leq L\bigl\| y-y'\bigr\|_{Y} ,\quad {\bf t}\in I,\ y\in R(F),\ y'\in c R(F).
\end{align}

The following result is an extension of \cite[Theorem 2.28]{c1}:

\begin{thm}\label{episkop-jovan}
Suppose that the functions $F : I \times X \rightarrow Y$ and $G : I \times Y \rightarrow Z$ are continuous as well as $\emptyset \neq I'\subseteq I \subseteq {\mathbb R}^{n}$ and \eqref{lajbaha} holds.
\begin{itemize}
\item[(i)]
Suppose further that, for every $B\in {\mathcal B}$ and $\epsilon>0,$
there exists $l>0$ such that for each ${\bf t}_{0} \in I'$ there exists ${\bf \tau} \in B({\bf t}_{0},l) \cap I'$ such that \eqref{ujshe} holds and
\begin{align}\label{ujshe1}
\bigl\|G({\bf t}+{\bf \tau};cy)-c F({\bf t};y)\bigr\|_{Z} \leq \epsilon,\quad {\bf t}\in I,\ y\in R(F).
\end{align}
Then the function $W(\cdot;\cdot),$ given by \eqref{skadar}, is Bohr $({\mathcal B},I',c)$-almost periodic.
\item[(ii)] 
Suppose further that, for every $B\in {\mathcal B}$, 
there exists a sequence $({\bf \tau}_{k})$ in $I'$ such that $\lim_{k\rightarrow +\infty} |{\bf \tau}_{k}|=+\infty,$ 
\eqref{ujshe-rec} holds and
\begin{align}\label{lajbaha-ujs}
\lim_{k\rightarrow +\infty}\sup_{{\bf t}\in I;x\in B} \bigl\|G\bigl({\bf t}+{\bf \tau}_{k};cF({\bf t}; x)\bigr)-cG({\bf t};F({\bf t}; x))\bigr\|_{Y} =0.
\end{align}
Then the function $W(\cdot;\cdot),$ given by \eqref{skadar}, is $({\mathcal B},I',c)$-uniformly recurrent.
\end{itemize}
\end{thm}

\begin{proof}
We will prove only (i). The continuity of function $W(\cdot;\cdot)$ is obvious.
Then the final conclusion follows from the assumption made, the corresponding definition of Bohr $({\mathcal B},I',c)$-almost periodicity and the next simple computation:
\begin{align*}
\bigl\| G\bigl({\bf t}&+\tau ; F({\bf t}+\tau; x)\bigr)-G\bigl({\bf t} ; F({\bf t}; x)\bigr)\bigr\|_{Z} 
\\& \leq \bigl\| G\bigl({\bf t}+\tau ; F({\bf t}+\tau; x)\bigr)-G\bigl({\bf t} +\tau; cF({\bf t}; x)\bigr)\bigr\|_{Z} 
\\& +\bigl\| G\bigl({\bf t}+\tau ; cF({\bf t}; x)\bigr)-cG\bigl({\bf t} ; F({\bf t}; x)\bigr)\bigr\|_{Z}
\\& \leq L\bigl\| F({\bf t}+\tau; x)-cF({\bf t}; x)\bigr\|_{Y} +\bigl\| G\bigl({\bf t}+\tau ; cF({\bf t}; x)\bigr)-cG\bigl({\bf t} ; F({\bf t}; x)\bigr)\bigr\|_{Z},
\end{align*}
for any ${\bf t}\in I,$ $\tau \in I'$ and $x\in X.$
\end{proof}

\subsection{${\mathbb D}$-asymptotically $({\mathcal B},c)$-almost periodic type functions}\label{gade-negadece}

In \cite{marko-manuel-ap}, we have recently introduced the following notion:

\begin{defn}\label{kompleks12345}
Suppose that 
${\mathbb D} \subseteq I \subseteq {\mathbb R}^{n}$ and the set ${\mathbb D}$  is unbounded. By $C_{0,{\mathbb D},{\mathcal B}}(I \times X :Y)$ we denote the vector space consisting of all continuous functions $Q : I \times X \rightarrow Y$ such that, for every $B\in {\mathcal B},$ we have $\lim_{t\in {\mathbb D},|t|\rightarrow +\infty}Q({\bf t};x)=0,$ uniformly for $x\in B.$\index{space!$C_{0,{\mathbb D}}(I \times X :Y)$}
\end{defn}

Now we are ready to introduce the following notion:

\begin{defn}\label{braindamage12345}
Suppose that the set ${\mathbb D} \subseteq I \subseteq {\mathbb R}^{n}$ is unbounded, $\emptyset \neq I'\subseteq I \subseteq {\mathbb R}^{n}$ and
$F : I \times X \rightarrow Y$ is a continuous function. Then we say that $F(\cdot ;\cdot)$ is 
(strongly) ${\mathbb D}$-asymptotically Bohr $({\mathcal B},I',c)$-almost periodic, resp. (strongly) ${\mathbb D}$-asymptotically $({\mathcal B},I',c)$-uniformly recurrent,
if and only if there exist a Bohr $({\mathcal B},I',c)$-almost periodic function ($G : {\mathbb R}^{n} \times X \rightarrow Y$) $G : I \times X \rightarrow Y$, resp. a $({\mathcal B},I',c)$-uniformly recurrent function ($G : {\mathbb R}^{n} \times X \rightarrow Y$)
$G : I \times X \rightarrow Y$
and a function
$Q\in C_{0,{\mathbb D},{\mathcal B}}(I\times X :Y)$ such that
$F({\bf t} ; x)=G({\bf t} ; x)+Q({\bf t} ; x)$ for all ${\bf t}\in I$ and $x\in X.$
If $I'=I,$ then we also say that $F(\cdot ;\cdot)$ is 
(strongly) ${\mathbb D}$-asymptotically Bohr $({\mathcal B},c)$-almost periodic, resp. (strongly) ${\mathbb D}$-asymptotically $({\mathcal B},c)$-uniformly recurrent; if $X\in {\mathcal B},$ then we omit the term ${\mathcal B}$
from the notation introduced, with the meaning clear.
\end{defn}

Before we go any further, we would like to present the following extension of \cite[Theorem 2.29]{c1}:

\begin{thm}\label{episkop-jovanas}
Suppose that the functions $F_{h} : I \times X \rightarrow Y,$ $F_{0} : I \times X \rightarrow Y,$ $G_{h} : I \times Y \rightarrow Z$and $G_{0} : I \times Y \rightarrow Z$ are continuous, $F=F_{h}+F_{0},$ $G=G_{h}+G_{0}$ as well as $\emptyset \neq I'\subseteq I \subseteq {\mathbb R}^{n}$ and \eqref{lajbaha} holds with the functions $F(\cdot ;\cdot)$ and $G(\cdot ;\cdot)$ replaced therein with the functions $F_{h}(\cdot ;\cdot)$ and $G_{h}(\cdot;\cdot),$ respectively.
\begin{itemize}
\item[(i)]
Suppose further that, for every $B\in {\mathcal B}$ and $\epsilon>0,$
there exists $l>0$ such that for each ${\bf t}_{0} \in I'$ there exists ${\bf \tau} \in B({\bf t}_{0},l) \cap I'$ such that \eqref{ujshe} holds with the function $F(\cdot ;\cdot)$ replaced with the function $F_{h}(\cdot;\cdot)$ and  \eqref{ujshe1} holds with the functions $F(\cdot ;\cdot)$ and $G(\cdot ;\cdot)$ replaced therein with the functions $F_{h}(\cdot ;\cdot)$ and $G_{h}(\cdot;\cdot),$ respectively.
If $F_{0}\in C_{0,{\mathbb D},{\mathcal B}}(I \times X :Y)$ and for each $B\in {\mathcal B}$ we have
$
\lim_{t\in {\mathbb D},|t|\rightarrow +\infty}G_{0}({\bf t};F({\bf t} ;x))=0, 
$
uniformly for $x\in B,$
then the function $W(\cdot;\cdot),$ given by \eqref{skadar}, is 
${\mathbb D}$-asymptotically Bohr $({\mathcal B},I',c)$-almost periodic.
\item[(ii)] 
Suppose further that, for every $B\in {\mathcal B}$, 
there exists a sequence $({\bf \tau}_{k})$ in $I'$ such that $\lim_{k\rightarrow +\infty} |{\bf \tau}_{k}|=+\infty,$ 
\eqref{ujshe-rec} holds and \eqref{lajbaha-ujs} holds with the functions $F(\cdot ;\cdot)$ and $G(\cdot ;\cdot)$ replaced therein with the functions $F_{h}(\cdot ;\cdot)$ and $G_{h}(\cdot;\cdot),$ respectively. If $F_{0}\in C_{0,{\mathbb D},{\mathcal B}}(I \times X :Y)$ and for each $B\in {\mathcal B}$ we have
$
\lim_{t\in {\mathbb D},|t|\rightarrow +\infty}G_{0}({\bf t};F({\bf t} ;x))=0, 
$
uniformly for $x\in B,$
then the function $W(\cdot;\cdot),$ given by \eqref{skadar}, is $({\mathcal B},I',c)$-uniformly recurrent.
\end{itemize}
\end{thm}

\begin{proof}
We will outline all details of the proof of (i) for the sake of completeness.
Clearly, the following decomposition holds true:
\begin{align*}
G(\cdot; F(\cdot;\cdot))=G_{h}\bigl(\cdot; F_{h}(\cdot;\cdot)\bigr)+\Bigl[G_{h}(\cdot; F(\cdot;\cdot))-G_{h}\bigl(\cdot; F_{h}(\cdot;\cdot)\bigr)\Bigr] +G_{0}(\cdot; F(\cdot;\cdot)).
\end{align*}
Due to Theorem \ref{episkop-jovan}, we have that the function $G_{h}(\cdot; F_{h}(\cdot;\cdot))$ is Bohr $({\mathcal B},I',c)$-almost periodic. Furthermore, the prescribed assumption implies that the function $G_{0}(\cdot; F(\cdot;\cdot))$ belongs to the space $C_{0,{\mathbb D},{\mathcal B}}(I \times X :Y).$ This also holds for the function $G_{h}(\cdot; F(\cdot;\cdot))-G_{h}(\cdot; F_{h}(\cdot;\cdot))$
since the function $G_{h}(\cdot;\cdot)$ satisfies the Lipschitz condition with respect to the first variable and $F_{0}\in C_{0,{\mathbb D},{\mathcal B}}(I \times X :Y).$
\end{proof}

Set, for brevity, $I_{{\bf t}}:=(-\infty,t_{1}] \times (-\infty,t_{2}]\times \cdot \cdot \cdot \times (-\infty,t_{n}]$ and 
${\mathbb D}_{{\bf t}}:=I_{{\bf t}} \cap {\mathbb D}$
for any ${\bf t}=(t_{1},t_{2},\cdot \cdot \cdot, t_{n})\in {\mathbb R}^{n}.$  Concerning the convolution invariance of strong ${\mathbb D}$-asymptotical $c$-almost periodicity under the actions of finite convolution products, we will formulate the following result (the proof is similar to the proof of corresponding result from \cite{marko-manuel-ap} and therefore omitted):

\begin{prop}\label{idio-multcq}
Suppose that $(R({\bf t}))_{{\bf t}> {\bf 0}}\subseteq L(X,Y)$ is a strongly continuous operator family such that
$\int_{(0,\infty)^{n}}\|R({\bf t} )\|\, d{\bf t}<\infty .$ If $f : I \rightarrow X$ is strongly ${\mathbb D}$-asymptotically $c$-almost periodic,
\begin{align*}
\lim_{|{\bf t}|\rightarrow \infty,  {\bf t} \in {\mathbb D}}\int_{I_{{\bf t}}\cap {\mathbb D}^{c}}\| R({\bf t}-{\bf s})\|\, d{\bf s}=0
\end{align*}
and for each $r>0$ we have
\begin{align*}
\lim_{|{\bf t}|\rightarrow \infty, {\bf t} \in {\mathbb D}}\int_{{\mathbb D}_{{\bf t}}\cap B(0,r)}\| R({\bf t}-{\bf s})\|\, d{\bf s}=0,
\end{align*}
then the function 
\begin{align*}
F({\bf t}):=\int_{{\mathbb D}_{{\bf t}}}R({\bf t}-{\bf s})f({\bf s})\, ds,\quad {\bf t}\in I
\end{align*}
is strongly ${\mathbb D}$-asymptotically $c$-almost periodic.
\end{prop}

Assuming that ${\mathbb D}=[\alpha_{1},\infty) \times [\alpha_{2},\infty) \times \cdot \cdot \cdot \times [\alpha_{n},\infty)$ for some real numbers $\alpha_{1},\ \alpha_{2},\cdot \cdot \cdot,\ \alpha_{n},$ then ${\mathbb D}_{{\bf t}}=[\alpha_{1},t_{1}]\times [\alpha_{2},t_{2}] \times \cdot \cdot \cdot \times [\alpha_{n},t_{n}].$ In this case, 
the function
$
F({\bf t})=\int^{{\bf \alpha}}_{{\bf t}}R({\bf t}-{\bf s})f({\bf s})\, ds,$ $ {\bf t}\in I
$ is strongly ${\mathbb D}$-asymptotically $c$-almost periodic, where we accept the notation
$$
\int^{{\bf \alpha}}_{{\bf t}}\cdot =\int_{\alpha_{1}}^{t_{1}}\int_{\alpha_{2}}^{t_{2}}\cdot \cdot \cdot \int_{\alpha_{n}}^{t_{n}}.
$$

Although clarified, we feel it is our duty to emphasize that our results concerning the invariance of multi-dimensional $c$-almost periodicity are not so easily applicable as the corresponding results known in the one-dimensional case, unfortunately. This is a very unexplored theme which will be further analyzed somewhere else.

Let
$F(\cdot;\cdot)$ be $I$-asymptotically $c$-uniformly recurrent, $G : I \times X \rightarrow Y$, 
$Q\in C_{0,I,{\mathcal B}}(I\times X :Y)$ and
$F({\bf t} ; x)=G({\bf t} ; x)+Q({\bf t} ; x)$ for all ${\bf t}\in I$ and $x\in X.$ Then, for every $x\in X$, we have
\begin{align*}
\overline{\bigl\{G({\bf t};x) : {\bf t}\in I,\ x\in X \bigr\}} \subseteq  \overline{\bigl\{F({\bf t};x) : {\bf t}\in I,\ x\in X\bigr\}}.
\end{align*}

The following proposition can be deduced as in the case that $c=1:$

\begin{prop}\label{mismodranini}
\begin{itemize}
\item[(i)]
Suppose that for each integer $j\in {\mathbb N}$ the function $F_{j}(\cdot ; \cdot)$ is  Bohr $({\mathcal B},c)$-almost periodic ($({\mathcal B},c)$-uniformly recurrent). If for each $B\in {\mathcal B}$ there exists $\epsilon_{B}>0$ such that
the sequence $(F_{j}(\cdot ;\cdot))$ converges uniformly to a function $F(\cdot ;\cdot)$ on the set $B^{\circ} \cup \bigcup_{x\in \partial B}B(x,\epsilon_{B}),$ then the function $F(\cdot ;\cdot)$ is Bohr $({\mathcal B},c)$-almost periodic ($({\mathcal B},c)$-uniformly recurrent).
\item[(ii)] Suppose that for each integer $j\in {\mathbb N}$ the function $F_{j}(\cdot ; \cdot)$ is $I$-asymptotically  Bohr $({\mathcal B},c)$-almost periodic ($I$-asymptotically $({\mathcal B},c)$-uniformly recurrent). If for each $B\in {\mathcal B}$ there exists $\epsilon_{B}>0$ such that
the sequence $(F_{j}(\cdot ;\cdot))$ converges uniformly to a function $F(\cdot ;\cdot)$ on the set $B^{\circ} \cup \bigcup_{x\in \partial B}B(x,\epsilon_{B}),$ then the function $F(\cdot ;\cdot)$ is $I$-asymptotically Bohr $({\mathcal B},c)$-almost periodic ($I$-asymptotically $({\mathcal B},c)$-uniformly recurrent).
\end{itemize}
\end{prop}

Now we will introduce the following definition (for any set $\Lambda \subseteq {\mathbb R}^{n}$ and number $M>0,$ we define $\Lambda_{M}:=\{ \lambda \in \Lambda \, ; \, |\lambda|\geq M \}$):

\begin{defn}\label{nafaks123456789012345123cea}
Suppose that 
${\mathbb D} \subseteq I \subseteq {\mathbb R}^{n}$ and the set ${\mathbb D}$ is unbounded, as well as
$\emptyset  \neq I'\subseteq I \subseteq {\mathbb R}^{n},$ $F : I \times X \rightarrow Y$ is a continuous function and $I +I' \subseteq I.$ Then we say that:
\begin{itemize}
\item[(i)]\index{function!${\mathbb D}$-asymptotically Bohr $({\mathcal B},I',c)$-almost periodic of type $1$}
$F(\cdot;\cdot)$ is ${\mathbb D}$-asymptotically Bohr $({\mathcal B},I',c)$-almost periodic  of type $1$ if and only if for every $B\in {\mathcal B}$ and $\epsilon>0$
there exist $l>0$ and $M>0$ such that for each ${\bf t}_{0} \in I'$ there exists ${\bf \tau} \in B({\bf t}_{0},l) \cap I'$ such that
\begin{align}\label{emojmarko145ce}
\bigl\|F({\bf t}+{\bf \tau};x)-cF({\bf t};x)\bigr\|_{Y} \leq \epsilon,\mbox{ provided } {\bf t},\ {\bf t}+\tau \in {\mathbb D}_{M},\ x\in B.
\end{align}
\item[(ii)] \index{function!${\mathbb D}$-asymptotically $({\mathcal B},I',c)$-uniformly recurrent  of type $1$}
$F(\cdot;\cdot)$ is ${\mathbb D}$-asymptotically $({\mathcal B},I',c)$-uniformly recurrent  of type $1$ if and only if for every $B\in {\mathcal B}$ 
there exist a sequence $({\bf \tau}_{k})$ in $I'$ and a sequence $(M_{k})$ in $(0,\infty)$ such that $\lim_{k\rightarrow +\infty} |{\bf \tau}_{k}|=\lim_{k\rightarrow +\infty}M_{k}=+\infty$ and
$$
\lim_{k\rightarrow +\infty}\sup_{{\bf t},{\bf t}+{\bf \tau}_{k}\in {\mathbb D}_{M_{k}};x\in B} \bigl\|F({\bf t}+{\bf \tau}_{k};x)-cF({\bf t};x)\bigr\|_{Y} =0.
$$
\end{itemize}
If $I'=I,$ then we also say that
$F(\cdot;\cdot)$ is ${\mathbb D}$-asymptotically Bohr $({\mathcal B},c)$-almost periodic of type $1$ (${\mathbb D}$-asymptotically $({\mathcal B},c)$-uniformly recurrent  of type $1$); furthermore, if $X\in {\mathcal B},$ then it is also said that $F(\cdot;\cdot)$ is ${\mathbb D}$-asymptotically Bohr $(I',c)$-almost periodic  of type $1$ (${\mathbb D}$-asymptotically $(I',c)$-uniformly recurrent  of type $1$). If $I'=I$ and $X\in {\mathcal B}$, then we also say that $F(\cdot;\cdot)$ is ${\mathbb D}$-asymptotically Bohr $c$-almost periodic of type $1$ (${\mathbb D}$-asymptotically $c$-uniformly recurrent of type $1$). As before, we remove the prefix ``${\mathbb D}$-'' in the case that ${\mathbb D}=I$ and remove the prefix ``$({\mathcal B},)$''  in the case that $X\in {\mathcal B}.$ 
\end{defn}

Clearly, we have the following:

\begin{prop}\label{okejecew}
Suppose that 
${\mathbb D} \subseteq I \subseteq {\mathbb R}^{n}$ and the set ${\mathbb D}$ is unbounded, as well as
$\emptyset  \neq I'\subseteq I \subseteq {\mathbb R}^{n},$ $F : I \times X \rightarrow Y$ is a continuous function and $I +I' \subseteq I.$ If  
$F(\cdot;\cdot)$ is ${\mathbb D}$-asymptotically Bohr $({\mathcal B},I',c)$-almost periodic, resp. ${\mathbb D}$-asymptotically $({\mathcal B},I',c)$-uniformly recurrent, then $F(\cdot;\cdot)$ is ${\mathbb D}$-asymptotically Bohr $({\mathcal B},I',c)$-almost periodic of type $1,$
resp. ${\mathbb D}$-asymptotically $({\mathcal B},I',c)$-uniformly recurrent of type $1$.
\end{prop}

Concerning the converse of Proposition \ref{okejecew}, we will state and prove the following statement which can be applied in the case that $I=[0,\infty)^{n}:$

\begin{thm}\label{bounded-paziemceq}
Suppose that $\emptyset  \neq I \subseteq {\mathbb R}^{n},$ $I +I =I,$ $I$ is closed and
$F : I \rightarrow Y$ is a uniformly continuous, bounded
$I$-asymptotically Bohr $c$-almost periodic function of type $1,$ where $|c|=1.$ If
\begin{align*}
\notag
(\forall l>0) \, (\forall M>0) \, (\exists {\bf t_{0}}\in I)\, (\exists k>0) &\, (\forall {\bf t} \in I_{M+l})(\exists {\bf t_{0}'}\in I)\,
\\ & (\forall {\bf t_{0}''}\in B({\bf t_{0}'},l) \cap I)\, {\bf t}- {\bf t_{0}''} \in B({\bf t_{0}},kl) \cap I_{M},
\end{align*}
there exists $L>0$ such that $I_{kL}\setminus I_{(k+1)L}  \neq \emptyset$ for all $k\in {\mathbb N}$ and $I_{M}+I\subseteq  I_{M}$ for all $M>0,$
then the function $F(\cdot)$ is $I$-asymptotically Bohr $c$-almost periodic.
\end{thm}

\begin{proof}
Since we have assumed that the function $F(\cdot)$ is bounded and $|c|=1,$ we can use the foregoing arguments in order to see that 
the function $F(\cdot)$ is $I$-asymptotically Bohr almost periodic function of type $1.$ By \cite[Theorem 2.34]{marko-manuel-ap}, 
it follows that for each sequence $({\bf b}_{k})$ in  
$I$ there exist a subsequence 
$({\bf b}_{k_{l}})$ of $({\bf b}_{k})$ and a function $F^{\ast} : I \rightarrow Y$ such that $\lim_{l\rightarrow +\infty}F({\bf t}+{\bf b}_{k_{l}})=F^{\ast}({\bf t}),$ uniformly in ${\bf t}\in I.$ We continue the proof by observing that
for each integer $k\in {\mathbb N}$  
there exist $l_{k}>0$ and $M_{k}>0$ such that for each ${\bf t}_{0} \in I$ there exists ${\bf \tau} \in B({\bf t}_{0},l) \cap I$ such that
\eqref{emojmarko145ce} holds with $\epsilon=1/k$ and ${\mathbb D}=I.$
Let ${\bf \tau}_{k}$ be any fixed element of $I$ such that $|{\bf \tau}_{k}|>M_{k}+k^{2}$ and \eqref{emojmarko145ce} holds with $\epsilon=1/k$ and ${\mathbb D}=I$ ($k\in {\mathbb N}$).
Then there exist of a subsequence $({\bf \tau}_{k_{l}})$ of $({\bf \tau}_{k})$ and a function
$F^{\ast} : I \rightarrow Y$
such that
\begin{align}\label{prce-franjo}
\lim_{l\rightarrow +\infty}F({\bf t}+{\bf \tau}_{k_{l}})=F^{\ast}({\bf t}),\mbox{ uniformly for }t\in I.
\end{align} 
The mapping $F^{\ast}(\cdot)$ is clearly continuous and now we will prove that $F^{\ast}(\cdot)$ is Bohr $c$-almost periodic. 
Let $\epsilon>0$ be fixed, and let $l>0$ and $M>0$ be such that for each ${\bf t}_{0} \in I$ there exists ${\bf \tau} \in B({\bf t}_{0},l) \cap I$ such that
\eqref{emojmarko145ce} holds with ${\mathbb D}=I$ and the number $\epsilon$ replaced therein by $\epsilon/3.$ Let ${\bf t }\in I$ be fixed, and let $l_{0}\in {\mathbb N}$ be such that $|{\bf t}+{\bf \tau}_{k_{l_{0}}}|\geq M$ and $|{\bf t}+{\bf \tau}+{\bf \tau}_{k_{l_{0}}}|\geq M.$ Then we have 
\begin{align*}
\Bigl\| & F^{\ast}({\bf t}+{\bf \tau})- cF^{\ast}({\bf t})\Bigr\|
\\& \leq \Bigl\|  F^{\ast}({\bf t}+{\bf \tau})- F\bigl({\bf t}+{\bf \tau}+{\bf \tau}_{k_{l_{0}}}\bigr)\Bigr\|+\Bigl\|  F\bigl({\bf t}+{\bf \tau}+{\bf \tau}_{k_{l_{0}}}\bigr)-c F\bigl({\bf t}+{\bf \tau}_{k_{l_{0}}}\bigr)\Bigr\|
\\&+
\Bigl\|  cF\bigl({\bf t}+{\bf \tau}_{k_{l_{0}}}\bigr)-cF^{\ast}({\bf t}) \Bigr\|\leq 3\cdot (\epsilon/3)=\epsilon,
\end{align*}
as required. The function ${\bf t} \mapsto F({\bf t} )-F^{\ast}({\bf t} ),$ ${\bf t} \in I$ belongs to the space $C_{0,I}(I: Y)$ due to \eqref{prce-franjo} and the fact that $F : I \rightarrow Y$ is an
$I$-asymptotically Bohr $c$-almost periodic function of type $1,$ which completes the proof.
\end{proof}

For any set $S\subseteq {\mathbb R}^{n}$ and for any integer $l\in {\mathbb N}$, we define the set $S_{l}$ inductively by $S_{1}:=S$ and $S_{l+1}:=S_{l}+S$ ($l=1,2, \cdot \cdot \cdot). $
Further on, we define  $\Omega:=I'$ and $\Omega_{S}:=I' \cup S$ if $\arg(c) /\pi \notin {\mathbb Q}.$ If $\arg(c) /\pi \in {\mathbb Q},$ then we
take any non-empty finite set of integers $S_{1}\subseteq {\mathbb Z} \setminus \{0\}$ such that $c^{m+1}=1$ for all $m\in S_{1}$ and any 
non-empty finite set of integers $S_{2}\subseteq {\mathbb N}$ such that $c^{l}=1$ for all $l\in S_{2};$
in this case, we set
$\Omega_:=(I'\bigcup_{m\in S_{1}} (-mI'))_{l}$ and $\Omega_{S}:=\Omega\cup S.$ 

Now we are able to state and prove the following result concerning the extensions of Bohr $(I',c)$-almost periodic functions and $(I',c)$-uniformly recurrent functions (in \cite[Proposition 2.25]{c1}, we have obeyed a different approach, where we have also considered semi-$c$-periodicity but not $c$-uniform recurrence): 

\begin{thm}\label{lenny-jassonce}
Suppose that $ I'\subseteq I \subseteq {\mathbb R}^{n},$ $I +I' \subseteq I,$ the set $I'$ is unbounded, $|c|=1,$ $F : I  \rightarrow Y$ is a uniformly continuous, Bohr $(I',c)$-almost periodic function, resp. a uniformly continuous, $(I',c)$-uniformly recurrent function, $S\subseteq {\mathbb R}^{n}$ is bounded and the following condition holds:
\begin{itemize}
\item[(AP-E)]\index{condition!(AP-E)}
For every ${\bf t}'\in {\mathbb R}^{n},$
there exists a finite real number $M>0$ such that
${\bf t}'+I'_{M}\subseteq I.$
\end{itemize}
Then there exists a uniformly continuous, Bohr $(\Omega_{S},c)$-almost periodic, resp. a uniformly continuous, $(\Omega_{S},c)$-uniformly recurrent, function
$\tilde{F} : {\mathbb R}^{n}  \rightarrow Y$ such that $\tilde{F}({\bf t})=F({\bf t})$ for all ${\bf t}\in I;$ furthermore, in $c$-almost periodic case, the uniqueness of such a function $\tilde{F}(\cdot)$ holds provided that ${\mathbb R}^{n} \setminus \Omega_{S}$ is a bounded set. 
\end{thm}

\begin{proof}
We will consider only uniformly continuous, Bohr $(I',c)$-almost periodic functions.
In this case,
for each natural number $k\in {\mathbb N}$  
there exists a point $\tau_{k}\in I'$ such that
$
\|F({\bf t}+{\bf \tau}_{k})-cF({\bf t})\|_{Y} \leq 1/k$ for all ${\bf t}\in I
$ and $k\in {\mathbb N};$ furthermore, since the set $I'$ is unbounded, we may assume without loss of generality that $\lim_{k\rightarrow +\infty}|{\bf \tau}_{k}|=+\infty.$ 
Hence, we have
\begin{align}\label{fuckfucka}
\lim_{k\rightarrow +\infty}F({\bf t}+{\bf \tau}_{k})=cF({\bf t}),\mbox{ uniformly for }t\in I.
\end{align}
If ${\bf t}'\in {\mathbb R}^{n},$ then there exists a finite real number $M>0$ such that
${\bf t}'+I'_{M}\subseteq I,$ and now we will prove that the sequence $(F({\bf t}'+{\bf \tau}_{k}))_{k\in {\mathbb N}}$
is Cauchy and therefore convergent. Let $\epsilon>0$ be fixed; then we have the existence of a number $k_{0}\in {\mathbb N}$
such that ${\bf t}'+{\bf \tau}_{k} \in I$ for all $k\geq k_{0}.$ Suppose that $k,\ m\geq k_{0}.$ Then we have
\begin{align*}
& \bigl\|  F({\bf t}'+{\bf \tau}_{k})-F({\bf t}'+{\bf \tau}_{m})\bigr\|\leq \bigl\| F({\bf t}'+{\bf \tau}_{k}) -c^{-1}F({\bf t}'+{\bf \tau}_{k}+\tau)\bigr\|
\\&+\bigl\|c^{-1}F({\bf t}'+{\bf \tau}_{k}+\tau) -c^{-1}F({\bf t}'+{\bf \tau}_{m}+\tau)\bigr\|+\bigl\| c^{-1}F({\bf t}'+{\bf \tau}_{m}+\tau) -F({\bf t}'+{\bf \tau}_{m})\bigr\|,
\end{align*}
for any $\tau \in I'$ such that ${\bf t}'+\tau \in I.$ Since the function $F(\cdot)$ is Bohr $(I',c)$-almost periodic, we can always find such a number $\tau$ so that the first and the third addend in the above estimates are less or equal than $\epsilon/3;$ for the second 
addend in the above estimate, we can
find a sufficiently large number $k_{1}\geq k_{0}$ such that 
$$
\bigl\|c^{-1}F({\bf t}'+{\bf \tau}_{k}+\tau) -c^{-1}F({\bf t}'+{\bf \tau}_{m}+\tau)\bigr\|<\epsilon/3,
$$
for all $k,\ m\geq k_{1}$ (see \eqref{fuckfucka}). Therefore, 
$\lim_{k\rightarrow +\infty}F({\bf t}'+{\bf \tau}_{k}):=\tilde{F}({\bf t}')$ exists.
The function $\tilde{F}(\cdot)$ is clearly uniformly continuous because $F(\cdot)$ is uniformly continuous; furthermore, by construction, we have that $\tilde{F}({\bf t})/c=F({\bf t})$ for all ${\bf t}\in I.$ 
Now we will prove that the function $\tilde{F}(\cdot)$
is Bohr $(\Omega_{S},c)$-almost periodic. Let a number $\epsilon>0 $ be given. Then there exists $l>0$ such that for each ${\bf t}_{0} \in I'$ there exists ${\bf \tau} \in B({\bf t}_{0},l) \cap I'$ such that
$
\|F({\bf t}+{\bf \tau})-cF({\bf t})\|_{Y} \leq \epsilon/2$ for all $ {\bf t}\in I.$ Let ${\bf t}'\in {\mathbb R}^{n}$ be fixed. For any such numbers ${\bf t}_{0} \in I'$ and ${\bf \tau} \in B({\bf t}_{0},l) \cap I',$ we have
\begin{align}
\notag \bigl\|\tilde{F}({\bf t}'&+{\bf \tau})-c\tilde{F}({\bf t}')\bigr\|_{Y} = \Bigl\|\lim_{k\rightarrow +\infty}\bigl[F({\bf t}'+{\bf \tau}+{\bf \tau}_{k})-cF({\bf t}'+{\bf \tau}_{k})\bigr]\Bigr\|_{Y} 
\\\label{valjevo} & \leq \limsup_{k\rightarrow +\infty}\bigl\|F({\bf t}'+{\bf \tau}+{\bf \tau}_{k})-cF({\bf t}'+{\bf \tau}_{k})\bigr\|_{Y}  
\leq \epsilon/2,\quad {\bf t}'\in {\mathbb R}^{n}.
\end{align}
If $\arg(c) /\pi \notin {\mathbb Q},$ this clearly implies 
that $F(\cdot)$ is Bohr $(\Omega,c)$-almost periodic and therefore Bohr ($\Omega_{S},c)$-almost periodic. If $\arg(c) /\pi \in {\mathbb Q},$ then we may assume without loss of generality that the sets $S_{1}=\{m\}$ and $S_{2}=\{l\}$ are singletons (this follows from the corresponding definition of Bohr $(I',c)$-almost periodicity). Given
$\epsilon>0 $ in advance, we may assume that \eqref{valjevo} holds with the number $\epsilon/2$ replaced therein with the number $\epsilon/l|m|.$
By \eqref{valevo}, we get that  the number $-m\tau \in \Omega$ is an $(\epsilon/l,c)$-period of $F(\cdot),$ with the meaning clear. Arguing as in the proof of the estimate \eqref{valevo}, it readily follows that any finite sum $\tau_{1}+\cdot \cdot \cdot +\tau_{l},$ where $\tau_{i}\in I'\bigcup_{m\in S_{1}} (-mI')$ for all $i\in {\mathbb N}_{l},$ is an $(\epsilon,c)$-period of $F(\cdot).$ As above, this implies that $F(\cdot)$ is Bohr $(\Omega,c)$-almost periodic and therefore Bohr ($\Omega_{S},c)$-almost periodic.

Assume, finally, that the set ${\mathbb R}^{n}\setminus \Omega_{S}$ is bounded. 
Then the function $\tilde{F}(\cdot)$ is Bohr $c$-almost periodic and bounded by Proposition \ref{bounded-pazice}; by the foregoing, this implies that the function $F(\cdot)$ is Bohr almost periodic and therefore compactly
almost automorphic. 
Then we can proceed as in the final part of the proof of \cite[Theorem 2.36]{c1} to prove the uniqueness of extension in $c$-almost periodic case. 
\end{proof}

\begin{rem}\label{minusplus1}
\begin{itemize}
\item[(i)]
It is clear that Theorem \ref{lenny-jassonce} strengthens \cite[Theorem 2.36]{marko-manuel-ap}, where we have assumed that $c=1$ and $\Omega_{S}=[(I'\cup (-I'))+(I'\cup (-I'))] \cup S.$ 
\item[(ii)] In the case that $\arg(c) /\pi \notin {\mathbb Q},$ it is not clear whether there exists a set $\Omega_{S}'\supseteq \Omega_{S}$ such that the constructed function $\tilde{F}: {\mathbb R}^{n} \rightarrow Y$ is Bohr $(\Omega_{S}',c)$-almost periodic. Concerning this
problematic, it is worth noting that the notion introduced in Definition \ref{nafaks123456789012345} can be further extended by allowing that the set $I'$ depends on the set $B$ and the number $\epsilon>0.$ This could probably fix some things here, but we will skip all related details for the sake of brevity.
\end{itemize}
\end{rem}

Before proceeding further, we would like to propose the following definition:

\begin{defn}\label{ujseadmiceqs}
Suppose that $\emptyset \neq I \subseteq {\mathbb R}^{n}$ and $I+I \subseteq I.$ Then we say that $I$ is admissible with respect to the $c$-almost periodic extensions if and only if for any complex Banach space $Y$ and for any 
uniformly continuous, Bohr $c$-almost periodic function $F : I\rightarrow Y$ there exists a unique Bohr $c$-almost periodic function $\tilde{F} : {\mathbb R}^{n} \rightarrow Y$ such that $\tilde{F}({\bf t})=F({\bf t})$ for all ${\bf t}\in I.$ 
If $c=\pm 1,$ then we also say that the region $I$ is admissible with respect to the almost (anti-)periodic extensions.
\end{defn}

If $|c|=1,$ $\arg(c) /\pi \in {\mathbb Q},$ $({\bf v}_{1},\cdot \cdot \cdot ,{\bf v}_{n})$ is a basis of ${\mathbb R}^{n}$ and
$$
I'=I=\bigl\{ \alpha_{1} {\bf v}_{1} +\cdot \cdot \cdot +\alpha_{n}{\bf v}_{n}  : \alpha_{i} \geq 0\mbox{ for all }i\in {\mathbb N}_{n} \bigr\}
$$ 
is a convex polyhedral in ${\mathbb R}^{n},$ then $\Omega_{S}={\mathbb R}^{n}$ and therefore the set $I$ is admissible with respect to the $c$-almost periodic extensions. It is very simple to construct some sets which are not 
admissible with respect to the $c$-almost periodic extensions; for example,
the set $I=[0,\infty) \times \{0\}\subseteq {\mathbb R}^{2}$ is not admissible with respect to the $c$-almost periodic extensions since there is no $c$-almost periodic extension
of the function $F(x,y)=y,$ $(x,y)\in I$ to the whole Euclidean space (\cite{c1}). 

\section{Examples and applications}\label{some12345}

In this section, we will present several interesting examples and applications of our abstract theoretical results. The first and second application have recently been considered in \cite{marko-manuel-ap}, with $c=1$:
\vspace{0.1cm}

1. Let $Y$ be one of the spaces $L^{p}({\mathbb R}^{n}),$ $C_{0}({\mathbb R}^{n})$ or $BUC({\mathbb R}^{n}),$ where $1\leq p<\infty.$ Then the Gaussian semigroup\index{Gaussian semigroup}
$$
(G(t)F)(x):=\bigl( 4\pi t \bigr)^{-(n/2)}\int_{{\mathbb R}^{n}}F(x-y)e^{-\frac{|y|^{2}}{4t}}\, dy,\quad t>0,\ f\in Y,\ x\in {\mathbb R}^{n},
$$
can be extended to a bounded analytic $C_{0}$-semigroup of angle $\pi/2,$ generated by the Laplacian $\Delta_{Y}$ acting with its maximal distributional domain in $Y;$ see \cite[Example 3.7.6]{a43}. Suppose 
that 
$\emptyset  \neq I'\subseteq I= {\mathbb R}^{n}$ and
$F(\cdot)$ is bounded Bohr $({\mathcal B},I',c)$-almost periodic, resp. bounded $({\mathcal B},I',c)$-uniformly recurrent. Then for each $t_{0}>0$ the function ${\mathbb R}^{n}\ni x\mapsto u(x,t_{0})\equiv (G(t_{0})F)(x) \in {\mathbb C}$
is likewise bounded Bohr $({\mathcal B},I',c)$-almost periodic, resp. bounded $({\mathcal B},I',c)$-uniformly recurrent. Towards this end, observe that for each $x,\,\ \tau \in {\mathbb R}^{n}$ we have:
$$
\Bigl|u\bigl(x+\tau,t_{0}\bigr)-cu\bigl(x,t_{0}\bigr)\Bigr| \leq \bigl( 4\pi t_{0} \bigr)^{-(n/2)}\int_{{\mathbb R}^{n}}|F(x-y+\tau)-cF(x-y)|e^{-\frac{|y|^{2}}{4t_{0}}}\, dy.
$$
see also Proposition \ref{convdiaggacece}.
We can  similarly clarify the corresponding results for the Poisson semigroup, which is analyzed in
\cite[Example 3.7.9]{a43}.

2. Define
$$
E_{1}(x,t):=\bigl( \pi t\bigr)^{-1/2}\int^{x}_{0}e^{-y^{2}/4t}\, dy,\quad x\in {\mathbb R},\ t>0
$$
and $I:=\{(x,t): x>0,\ t>0\}$. Recall that F. Tr\`eves \cite[p. 433]{treves} has proposed the following formula:
\begin{align}\label{out-zxc}
u(x,t)=\frac{1}{2}\int^{x}_{-x}\frac{\partial E_{1}}{\partial y}(y,t)u_{0}(x-y)\, dy-\int^{t}_{0}\frac{\partial E_{1}}{\partial t}(x,t-s)g(s)\, ds,\quad x>0,\ t>0,
\end{align}
for the solution of the following mixed initial value problem:
\begin{align}\label{heat-prvi}
\begin{split}
& u_{t}(x,t)=u_{xx}(x,t),\ \ x>0,\ t>0; \\
& u(x,0)=u_{0}(x), \ x>0,\ \  u(0,t)=g(t), \  t>0.
\end{split}
\end{align}
Consider the case in which $g(t)\equiv 0.$ Suppose that $0<T<\infty$ and the function $u_{0} : [0,\infty) \rightarrow {\mathbb C}$ is bounded Bohr $(I_{0},c)$-almost periodic, resp. bounded $(I_{0},c)$-uniformly recurrent, for a certain non-empty subset $I_{0}$ of $[0,\infty).$
Set $I':=I_{0} \times (0,T).$ If ${\mathbb D}$ is any unbounded subset of $I$ which has the property that 
$$
\lim_{|(x,t)| \rightarrow +\infty, (x,t)\in {\mathbb D}}\min\Biggl(\frac{x^{2}}{4(t+T)},t\Biggr)=+\infty,  
$$
then the solution $u(x,t)$ of \eqref{heat-prvi} is ${\mathbb D}$-asymptotically $(I',c)$-almost periodic of type $1$, resp. ${\mathbb D}$-asymptotically $(I',c)$-uniformly recurrent of type $1$. This can be achieved by a careful inspection of the argumentation given in \cite[Section 3, point 2.]{marko-manuel-ap}.

3. (cf. also \cite[Theorem 3.1]{c1})
Let $(\tau_{k})$ be a  sequence in ${\mathbb R}^{n},$ $\lim_{k\rightarrow +\infty}|\tau_{k}|=+\infty$ and
\begin{align*}
BUR_{(\tau_{k});c}({\mathbb R}^{n} : X)&:=\Bigl\{F : {\mathbb R}^{n} \rightarrow X\mbox{ is bounded, continuous and } 
\\&
\lim_{k\rightarrow+\infty}\sup_{t\in {\mathbb R}}\bigl\| F(t+\tau_{k})-cf(t)\bigr\|_{\infty}=0
\Bigr\}.
\end{align*}
Equipped with the metric $d(\cdot,\cdot):=\|\cdot-\cdot\|_{\infty},$ $BUR_{(\tau_{k});c}({\mathbb R}^{n} : X)$ becomes a complete metric space. Define $I':=\{\tau_{k} : k\in {\mathbb N}\}$ 
and consider
the following Hammerstein integral equation of convolution type on ${\mathbb R}^{n}$ (see e.g., \cite[Section 4.3, pp. 170-180]{cord-int}):
\begin{align}\label{multiintegralce}
y({\bf t})= \int_{{\mathbb R}^{n}}k({\bf t}-{\bf s})G({\bf s},y({\bf s}))\, d{\bf s},\quad {\bf t}\in {\mathbb R}^{n},
\end{align}
where $G : {\mathbb R}^{n} \times X  \rightarrow X$ is $({\mathcal B},I',c)$-uniformly recurrent with ${\mathcal B}$ being the collection of all bounded subsets of $X.$
Suppose, further, that the set $\{G({\bf t} , B) : {\bf t}\in {\mathbb R}^{n}\}$ is bounded for any bounded subset $B$ of $X$ as well as that 
there exists a finite real constant $L>0$ such that \eqref{lajbaha} holds with $X=Z=Y,$ for every $y,\ y'\in {\mathbb R}^{n},$
and
\eqref{lajbaha-ujs} holds with 
the term $F({\bf t} ; x)$ replaced with the term $y({\bf t})$ for any function $y\in BUC_{(\tau_{k});c}({\mathbb R}^{n} : X).$ 
Applying 
Proposition \ref{convdiaggacece} and Theorem \ref{episkop-jovan}(ii), we get that the mapping 
$$
BUR_{(\tau_{k});c}({\mathbb R}^{n} : X) \ni y \mapsto \int_{{\mathbb R}^{n}}k({\cdot}-{\bf s})G({\bf s},y({\bf s}))\, d{\bf s} \in BUR_{(\tau_{k});c}({\mathbb R}^{n} : X)
$$
is well defined. If we additionally assume that $L\int_{{\mathbb R}^{n}}|k({\bf t})| \, d{\bf t}<1,$ then an application of
the Banach contraction principle
shows that there exists a unique solution of \eqref{multiintegralce} which belongs to the space $BUR_{(\tau_{k});c}({\mathbb R}^{n} : X).$

We can similarly analyze the following integral equation
\begin{align*}
y({\bf t})= \int_{{\mathbb R}^{n}}G({\bf t},{\bf s},y({\bf s}))\, d{\bf s},\quad {\bf t}\in {\mathbb R}^{n},
\end{align*}
provided that $G: {\mathbb R}^{2n} \times X  \rightarrow X$ satisfies certain assumptions and
there exists a costant $L\in (0,1)$ such that
$$
\| G({\bf t},{\bf s},x)-G({\bf t},{\bf s},y)\| \leq L\|x-y\|,\quad {\bf t}, \ {\bf s}\in {\mathbb R}^{n}; \ x, \ y\in X.
$$
Details can be left to the interested readers.

We close the paper with the observation that the class of
multi-dimensional $({\bf \omega},c)$-periodic type functions will be considered in our forthcoming paper \cite{nova-msemice}.

\end{document}